\documentclass[12pt,a4paper]{article}
\usepackage[utf8]{inputenc}
\usepackage{amsmath}
\usepackage{breqn}
\usepackage{amsthm}
\usepackage{amsfonts}
\usepackage{amssymb} 
\usepackage{graphicx}
\usepackage[left=2cm,right=2cm,top=2cm,bottom=2cm]{geometry}
\usepackage[numbers]{natbib}
\usepackage{lineno,hyperref}
\modulolinenumbers[5]

\newtheorem{theorem}{Theorem}[section]
\newtheorem{corollary}{Corollary}[section]
\newtheorem{lemma}{Lemma}[section]
\newtheorem{proposition}{Proposition}[section]

\newtheorem*{keyword}{Keywords}
\newtheorem*{remark}{Remark}
\newtheorem{example}{Example}[section]
\newtheorem*{definition}{Definition}

\newcommand*{\Scale}[2][4]{\scalebox{#1}{$#2$}}

\begin{document}
\title{Recurrent Sums and Partition Identities}
\author{Roudy El Haddad \\
Universit\'e La Sagesse, Facult\'e de g\'enie, Polytech }
\maketitle

\begin{abstract}
Sums of the form $\sum_{N_m=q}^{n}{\cdots \sum_{N_1=q}^{N_2}{a_{(m);N_m}\cdots a_{(1);N_1}}}$ where the $a_{(k);N_k}$'s are same or distinct sequences appear quite often in mathematics. We will refer to them as \textit{recurrent sums}. In this paper, we introduce a variety of formulas to help manipulate and work with this type of sums. We begin by developing variation formulas that allow the variation of a recurrent sum of order $m$ to be expressed in terms of lower order recurrent sums. We then proceed to derive theorems (which we will call inversion formulas) which show how to interchange the order of summation in a multitude of ways. Later, we introduce a set of new \textit{partition identities} in order to then prove a reduction theorem which permits the expression of a recurrent sum in terms of a combination of non-recurrent sums. Finally, we apply this reduction theorem to a recurrent form of two famous types of sums: The $p$-series and the sum of powers. 
\end{abstract}

\begin{keyword}
{\em Recurrent sums, Partitions, Stirling numbers of the first kind, Bell polynomials, Multiple harmonic series, Riemann zeta function, Bernoulli numbers, Faulhaber formula. } \\
\end{keyword}

\section{Introduction and Notation} \label{intro}
The harmonic series was first studied and proven to diverge in the 14th century by Nicole Oresme \cite{Oresme}. Later, in the 17th century, new proofs for this divergence were provided by Pietro Mengoli \cite{Mengoli}, Johann Bernoulli \cite{JohannBernoulli}, and Jacob Bernoulli \cite{JacobBernoulli1,JacobBernoulli2}. 
However, a more general form of this series does converge. Euler was the first to study such sums of the form 
$$
\zeta(s)=\sum_{n=1}^{\infty}{\frac{1}{n^s}}
$$ 
where $s$ is a real number. In the famous Basel problem, Euler proved that $\zeta(2)=\frac{\pi^2}{6}$ (see \cite{Basel1,Basel2,Basel3}. Fourteen additional proofs can be found in \cite{Basel4}). He later provided a general formula for this zeta function for positive even values of $s$. 
Euler's definition was then extended to a complex variable $s$ by Riemann in his 1859 article ``On the Number of Primes Less Than a Given Magnitude''. More recently, the multiple harmonic series, an even more general form of the zeta function, has been introduced and studied. Note that Euler was the first to study these multiple harmonic series for length 2 in \cite{euler}. 
A multiple harmonic series (MHS) or multiple zeta values (MZV) is defined as: 
$$
\zeta(s_1,s_2, \ldots, s_k)
=\sum_{1 \leq N_1 < N_2 < \cdots < N_k}{\frac{1}{N_1^{s_1}N_2^{s_2}\cdots N_k^{s_k}}}
.$$
A very important variant of the MHS (see \cite{kuba2019,hoffman1996sums,murahara2019interpolation}) often referred to as multiple zeta star values MZSV or multiple harmonic star series MHSS (or simply multiple zeta values) is defined by: 
$$
\zeta^{\star}(s_1,s_2, \ldots, s_k)
=\sum_{1 \leq N_1 \leq N_2 \leq \cdots \leq N_k}{\frac{1}{N_1^{s_1}N_2^{s_2}\cdots N_k^{s_k}}}
.$$
This variant of the multiple harmonic series is directly related to the Riemann zeta function $\zeta(s)$ \cite{hoffman1992multiple,granville1997decomposition}. Additionally, it is involved in a variety of sums and series including the Arakawa–Kaneko zeta function \cite{xuduality} and Euler sums. 

Such sums have tremendous importance in number theory.
They have been of interest to mathematicians for a long time and have been systematically studied since the 1990s with the work of Hoffman \cite{hoffman1992multiple,HOFFMAN1997} and Zagier \cite{zagier1994values}. However, their importance is not limited to Number Theory. In fact, such sums/series have appeared in physics even before the phrase “multiple zeta values” had been coined. As an example, the number $\zeta(\overline{6},\overline{2})$ appeared in the quantum field theory literature in 1986 \cite{broadhurst1986exploiting}. They play a major role in the connection of knot theory with quantum field theory \cite{broadhurst2013multiple,kassel2012quantum}. MZVs and MZSVs became even more important after they became needed for higher order calculations in quantum electrodynamics (QED) and quantum chromodynamics (QCD) \cite{blumlein1999harmonic,blumlein2010multiple}. 

These sums are a particular case of what we called recurrent sums as they are of the form $\sum_{1\leq N_1 \leq \cdots \leq N_m \leq n}{a_{(m);N_m} \cdots a_{(1);N_1}}$ with $a_{(i);N_i}=\frac{1}{N_i^{s_i}}$ for all $i$. The particular case has been extensively studied while the general case received much less interest. Although there are hundreds if not thousands of formulae to help in the study of multiple harmonic star sums and multiple zeta star values, barely any formulae can be found for its general counterpart. In this article, we are interested in studying this more general form which is expressed as follows: 
$$\sum_{1\leq N_1 \leq \cdots \leq N_m \leq n}{a_{(m);N_m} \cdots a_{(1);N_1}}.$$
We will also consider the particular case where all sequences are the same: 
$$\sum_{1\leq N_1 \leq \cdots \leq N_m \leq n}{a_{N_m} \cdots a_{N_1}}.$$ 
This structure of sums appears in a variety of areas of mathematics. The objective is to develop formulae to improve and facilitate the way we work with recurrent sums. This includes deriving formulae to calculate the variation of such sums, formulae to interchange the order of summation as well as formulae to represent recurrent sums in terms of a combination of non-recurrent sums. Note that this type of sums is intimately related to partitions as they appear in the representation of recurrent sums as a combination of simple non-recurrent sums. Therefore, this article will also focus on partition identities that are needed to prove the previously stated theorems as well as the ones that can be derived from these same theorems. Among these partition identities that can be found through these theorems, a definition of binomial coefficients in terms of a sum over partitions will be presented. Similarly, we produce some identities involving special sums, over partitions, of Bernoulli numbers. Furthermore, we are also interested in applying the formulae develop for the general case to some particular cases. First, we will apply our results to the multiple sums of powers in order to generalize Faulhaber's formula. 
Then, we will go back to the most famous particular case which is the MZSV and show how our results on the general case can improve in this case. 
A particularly beautiful identity that we will present is the following which relates the recurrent sum of $\frac{1}{N^2}$ to the zeta function for positive even values:  
$$\sum_{N_m=1}^{\infty}{\cdots \sum_{N_1=1}^{N_2}{\frac{1}{N_m^2\cdots N_1^2}}}
=\sum_{\sum{i.y_{k,i}}=m}{\prod_{i=1}^{m}{\frac{1}{(y_{k,i})!i^{y_{k,i}}} \left(\zeta(2i)\right)^{y_{k,i}}}}
=\left(2-\frac{1}{2^{2(m-1)}}\right)\zeta(2m).
$$
Although this paper focuses on the generalized version of the multiple zeta star values, the multiple zeta values itself is a particular form of a type of sums presented in \cite{PolynomialSums} and which is closely related to the recurrent sums by the relations also presented in that cited article.

The main theorems of this paper have potential applications such as the following: Surprisingly, this form appears in the general formula for the $n$-th integral of $x^m(\ln x)^{m'}$. In the unpublished paper \cite{nIntegral}, the relations presented in this paper will be used to derive and prove this general formula for the $n$-th integral of $x^m(\ln x)^{m'}$. 
In paper \cite{PolynomialSums}, the partition identities here presented are combined with additional partition identities in order to produce identities for odd and even partitions.

Let us now introduce some notation in order to facilitate the representation of such sums in this paper: For any $m,q,n \in \mathbb{N}$ where $n \geq q$ and for any set of sequences $a_{(1);N_1},\cdots,a_{(m);N_m}$ defined in the interval $[q,n]$, let $R_{m,q,n}(a_{(1);N_1},\cdots,a_{(m);N_m})$ represent the general recurrent sum of order $m$ for the sequences $a_{(1);N_1},\cdots,a_{(m);N_m}$ with lower an upper bounds respectively $q$ and $n$. For simplicity, however, we will denote it simply as $R_{m,q,n}$.
\begin{equation} \label{eq1}
\begin{split}
R_{m,q,n}
&=\sum_{N_m=q}^{n}{a_{(m);N_m} \cdots \sum_{N_2=q}^{N_3}{a_{(2);N_2}\sum_{N_1=q}^{N_2}{a_{(1);N_1}}}}\\
&=\sum_{N_m=q}^{n}{\cdots \sum_{N_2=q}^{N_3}{\sum_{N_1=q}^{N_2}{a_{(m);N_m}\cdots a_{(2);N_2}a_{(1);N_1}}}} \\
&=\sum_{q\leq N_1 \leq \cdots \leq N_m \leq n}{a_{(m);N_m}\cdots a_{(2);N_2}a_{(1);N_1}}.\\
\end{split}
\end{equation}
The most common case of a recurrent sum is that where all sequences are the same, 
\begin{equation} \label{eq2}
\begin{split}
R_{m,q,n}(a_{N_1} \cdots a_{N_m})
&=\sum_{N_m=q}^{n}{a_{N_m} \cdots \sum_{N_2=q}^{N_3}{a_{N_2}\sum_{N_1=q}^{N_2}{a_{N_1}}}}\\
&=\sum_{N_m=q}^{n}{\cdots \sum_{N_2=q}^{N_3}{\sum_{N_1=q}^{N_2}{a_{N_m}\cdots a_{N_2}a_{N_1}}}} \\
&=\sum_{q\leq N_1 \leq \cdots \leq N_m \leq n}{a_{N_m}\cdots a_{N_2}a_{N_1}}.\\
\end{split}
\end{equation}
For simplicity, we will denote it as $\hat{R}_{m,q,n}$. \\
This type of sums is described as recurrent because they can also be expressed using the following recurrent form: 
\begin{equation}
\begin{cases}
R_{m,q,n}=\sum_{N_m=q}^{n}{a_{(m);N_m}R_{m-1,q,N_m}}\\
R_{0,i,j}\,\,\,=1 \,\, \forall i,j \in \mathbb{N}.
\end{cases}
\end{equation}
\begin{remark}
{\em A recurrent sum of order $0$ is always equal to $1$. It is not equivalent to an empty sum $($which is equal to $0)$.}
\end{remark} 
In this paper, this type of sums will be studied. In Section \ref{Variation Formulas}, formulas for the calculation of variation of these sums in terms of lower order recurrent sums will be presented. Then, in Section \ref{Inversion Formulas}, inversion formulas will be presented, which will allow the interchange of the order of summation in such sums. Finally, in Section \ref{Reduction Formulas}, we will present a reduction formula that allows the representation of a recurrent sum as a combination of simple (non-recurrent) sums. These relations will be, then, used to calculate certain special sums such as the recurrent harmonic sum and the recurrent equivalent to the Faulhaber formula.    

\section{Variation Formulas} \label{Variation Formulas}
In this section, we will develop formulas to express the variation of a recurrent sum of order $m$ ($R_{m,q,n+1}-R_{m,q,n}$) in terms of lower order recurrent sums. Equivalently, these formulas can be used to express $R_{m,q,n+1}$ in terms of $R_{m,q,n}$ and lower order recurrent sums. 
\subsection{Simple expression}
We start by proving the most basic form for the variation formula as illustrated by the following Lemma. This is needed in order to prove the general form of this formula. 
\begin{lemma} \label{Lemma 2.1}
For any $m,q,n \in \mathbb{N}$, we have that 
$$R_{m+1,q,n+1}=a_{(m+1);n+1} R_{m,q,n+1}+R_{m+1,q,n}.$$
\end{lemma}
\begin{proof}
\begin{equation*}
\begin{split}
R_{m+1,q,n+1} 
& = \sum_{N_{m+1}=q}^{n+1}{\cdots \sum_{N_1=q}^{N_2}{a_{(m+1);N_{m+1}}\cdots a_{(1);N_1}}} \\
& = a_{(m+1);n+1} \sum_{N_m=q}^{n+1}{\cdots \sum_{N_1=q}^{N_2}{a_{(m);N_m}\cdots a_{(1);N_1}}}+\sum_{N_{m+1}=q}^{n}{\cdots \sum_{N_1=q}^{N_2}{a_{(m+1);N_{m+1}}\cdots a_{(1);N_1}}} \\
& = a_{(m+1);n+1} R_{m,q,n+1}+R_{m+1,q,n}.
\end{split}
\end{equation*}
\end{proof}
Now we apply the basic case from Lemma \ref{Lemma 2.1} to show the general variation formula that allows $R_{m,q,n+1}$ to be expressed in terms of $R_{m,q,n}$ and of recurrent sums of order going from $0$ to $(m-1)$. 
\begin{theorem} \label{Theorem 2.1}
For any $m,q,n \in \mathbb{N}$ where $n \geq q$ and for any set of sequences $a_{(1);N_1},\cdots,a_{(m);N_m}$ defined in the interval $[q,n+1]$, we have that 
$$\sum_{N_m=q}^{n+1}{\cdots \sum_{N_1=q}^{N_2}{a_{(m);N_m}\cdots a_{(1);N_1}}}
=\sum_{k=0}^{m}{\left(\prod_{j=0}^{m-k-1}{a_{(m-j);n+1}}\right)\left(\sum_{N_k=q}^{n}{\cdots \sum_{N_1=q}^{N_2}{a_{(k);N_k}\cdots a_{(1);N_1}}}\right)}.$$
Using the notation from Eq.~\eqref{eq1}, this theorem can be written as 
$$
R_{m,q,n+1}=\sum_{k=0}^{m}{\left(\prod_{j=0}^{m-k-1}{a_{(m-j);n+1}}\right)R_{k,q,n}}.
$$
\end{theorem}
\begin{proof}
$ \\ $
1. Base Case: verify true for $m=1$.
\begin{equation*}
\begin{split}
\sum_{k=0}^{1}{\left(\prod_{j=0}^{-k}{a_{(1-j);n+1}}\right)R_{k,q,n}}
&=\left(\prod_{j=0}^{0}{a_{(1-j);n+1}}\right)R_{0,q,n}
+\left(\prod_{j=0}^{-1}{a_{(1-j);n+1}}\right)R_{1,q,n}\\
&=(a_{(1);n+1})(1)+(1)\left(\sum_{N_1=q}^{n}{a_{(1);N_1}}\right)\\
&=R_{1,q,n+1}.
\end{split}
\end{equation*}
2. Induction hypothesis: assume the statement is true until $m$. 
$$
R_{m,q,n+1}=\sum_{k=0}^{m}{\left(\prod_{j=0}^{m-k-1}{a_{(m-j);n+1}}\right)R_{k,q,n}}
.$$
3. Induction step: we will show that this statement is true for ($m+1$).\\
We have to show the following statement to be true: 
$$
R_{m+1,q,n+1}=\sum_{k=0}^{m+1}{\left(\prod_{j=0}^{m-k}{a_{(m+1-j);n+1}}\right)R_{k,q,n}}
.$$
$ \\ $
From Lemma \ref{Lemma 2.1}, 
$$R_{m+1,q,n+1}=a_{(m+1);n+1} R_{m,q,n+1}+R_{m+1,q,n}.$$
By applying the induction hypothesis, 
\begin{equation*}
\begin{split}
R_{m+1,q,n+1}
&=a_{(m+1);n+1}\sum_{k=0}^{m}{\left(\prod_{j=0}^{m-k-1}{a_{(m-j);n+1}}\right)R_{k,q,n}}
+R_{m+1,q,n}\\
&=a_{(m+1);n+1}\sum_{k=0}^{m}{\left(\prod_{j=1}^{m-k}{a_{(m+1-j);n+1}}\right)R_{k,q,n}}
+R_{m+1,q,n}\\
&=\sum_{k=0}^{m}{\left(\prod_{j=0}^{m-k}{a_{(m+1-j);n+1}}\right)R_{k,q,n}}+R_{m+1,q,n}.
\end{split}
\end{equation*}
Noticing that 
$$
\sum_{k=m+1}^{m+1}{\left(\prod_{j=0}^{m-k}{a_{(m+1-j);n+1}}\right)R_{k,q,n}}
=R_{m+1,q,n}
$$
hence, 
$$
R_{m+1,q,n+1}
=\sum_{k=0}^{m+1}{\left(\prod_{j=0}^{m-k}{a_{(m+1-j);n+1}}\right)R_{k,q,n}}
.$$
Hence, the theorem is proven by induction. 
\end{proof}
\begin{corollary} \label{Corollary 2.1}
If all sequences are the same, Theorem \ref{Theorem 2.1} will be reduced to the following form, 
$$\sum_{N_m=q}^{n+1}{\cdots \sum_{N_1=q}^{N_2}{a_{N_m}\cdots a_{N_1}}}=\sum_{k=0}^{m}{\left(a_{n+1}\right)^{m-k}\left(\sum_{N_k=q}^{n}{\cdots \sum_{N_1=q}^{N_2}{a_{N_k}\cdots a_{N_1}}}\right)}.$$
Using the notation from Eq.~\eqref{eq2}, this theorem can be written as 
$$
\hat{R}_{m,q,n+1}=\sum_{k=0}^{m}{\left(a_{n+1}\right)^{m-k}\hat{R}_{k,q,n}}.
$$
\end{corollary}
\begin{example}
{\em Consider that $m=2$, we have the two following cases: }
\begin{itemize}
\item {\em If all sequences are distinct, } 
$$
\sum_{N_2=q}^{n+1}{b_{N_2} \sum_{N_1=q}^{N_2}{a_{N_1}}}
-\sum_{N_2=q}^{n}{b_{N_2} \sum_{N_1=q}^{N_2}{a_{N_1}}}
=(b_{n+1})\sum_{N_1=q}^{n}{a_{N_1}}+(b_{n+1})(a_{n+1})
.$$
\item {\em If all sequences are the same,  }
$$
\sum_{N_2=q}^{n+1}{a_{N_2} \sum_{N_1=q}^{N_2}{a_{N_1}}}
-\sum_{N_2=q}^{n}{a_{N_2} \sum_{N_1=q}^{N_2}{a_{N_1}}}
=(a_{n+1})\sum_{N_1=q}^{n}{a_{N_1}}+(a_{n+1})^2
.$$
\end{itemize}
\end{example}
\begin{remark}
{\em Set $a_{(m);N}=\cdots = a_{(2);N}=1$, Theorem \ref{Theorem 2.1} becomes }
$$
\sum_{N_m=q}^{n+1}{\sum_{N_{m-1}=q}^{N_m}{\cdots \sum_{N_1=q}^{N_2}{a_{N_1}}}}
=\sum_{k=1}^{m}{\left( \sum_{N_k=q}^{n}{\sum_{N_{k-1}=q}^{N_k}{\cdots \sum_{N_1=q}^{N_2}{a_{N_1}}}}\right)}+a_{n+1}.
$$
\end{remark}
\subsection{Simple recurrent expression}
A recursive form of Theorem \ref{Theorem 2.1} can be obtained by expanding and factoring the theorem's expression.  
\begin{theorem} \label{Theorem 2.2}
For any $m,q,n \in \mathbb{N}$ where $n \geq q$ and for any set of sequences $a_{(1);N_1},\cdots,a_{(m);N_m}$ defined in the interval $[q,n+1]$, we have that 
\begin{dmath*}
\sum_{N_m=q}^{n+1}{\cdots \sum_{N_1=q}^{N_2}{a_{(m);N_m}\cdots a_{(1);N_1}}}
-\sum_{N_m=q}^{n}{\cdots \sum_{N_1=q}^{N_2}{a_{(m);N_m}\cdots a_{(1);N_1}}}
=a_{(m);n+1}\left\{a_{(m-1);n+1} \left[ \cdots a_{(2);n+1} \left( a_{(1);n+1}(1)+\sum_{N_1=q}^{n}{a_{(1);N_1}} \right)+\sum_{N_2=q}^{n}{ \sum_{N_1=q}^{N_2}{a_{(2);N_2} a_{(1);N_1}}}\right] +\sum_{N_{m-1}=q}^{n}{\cdots \sum_{N_1=q}^{N_2}{a_{(m-1);N_{m-1}}\cdots a_{(1);N_1}}} \right\}.
\end{dmath*}
Using the notation from Eq.~\eqref{eq1}, this theorem can be written as 
\begin{dmath*}
R_{m,q,n+1}=a_{(m);n+1}\left\{ a_{(m-1);n+1}\left[ \cdots a_{(2);n+1} \left( a_{(1);n+1} \left( R_{0,q,n} \right) + R_{1,q,n} \right) + R_{2,q,n} \right] + R_{m-1,q,n} \right\} + R_{m,q,n}
\end{dmath*}
where $R_{0,q,n}=1$. 
\end{theorem}
\begin{proof}
$ \\ $
1. Base Case: verify true for $m=1$. 
$$
a_{(1);n+1}(R_{0,q,n})+R_{1,q,n}
=a_{(1);n+1}(1)+\sum_{N_1=q}^{n}{a_{(1);N_1}}
=\sum_{N_1=q}^{n+1}{a_{(1);N_1}}
=R_{1,q,n+1}
.$$
2. Induction Hypothesis: assume the statement is true until $m$.
\begin{dmath*}
\Scale[0.9]{
R_{m,q,n+1}=a_{(m);n+1}\left\{ a_{(m-1);n+1}\left[ \cdots a_{(2);n+1} \left( a_{(1);n+1} \left( R_{0,q,n} \right) + R_{1,q,n} \right) + R_{2,q,n} \right] + R_{m-1,q,n} \right\} + R_{m,q,n}.}
\end{dmath*}
3. Induction Step: we will show that this statement is true for $(m+1)$. \\
We have to show the following statement to be true: 
\begin{dmath*}
\Scale[0.9]{
R_{m+1,q,n+1}=a_{(m+1);n+1}\left\{ a_{(m);n+1}\left[ \cdots a_{(2);n+1} \left( a_{(1);n+1} \left( R_{0,q,n} \right) + R_{1,q,n} \right) + R_{2,q,n} \right] + R_{m,q,n} \right\} + R_{m+1,q,n}.}
\end{dmath*}
$ \\ $ 
From Lemma \ref{Lemma 2.1}, 
$$R_{m+1,q,n+1}=a_{(m+1);n+1} R_{m,q,n+1}+R_{m+1,q,n}.$$ 
By applying the induction hypothesis, 
\begin{dmath*}
\Scale[0.9]{
R_{m+1,q,n+1}=a_{(m+1);n+1}\left\{ a_{(m);n+1}\left[ \cdots a_{(2);n+1} \left( a_{(1);n+1} \left( R_{0,q,n} \right) + R_{1,q,n} \right) + R_{2,q,n} \right] + R_{m,q,n} \right\} + R_{m+1,q,n}. }
\end{dmath*} 
Hence, the theorem is proven by induction. 
\end{proof}
\begin{corollary} \label{Corollary 2.3}
If all sequences are the same, Theorem \ref{Theorem 2.2} will be reduced to the following form, 
\begin{dmath*}
\sum_{N_m=q}^{n+1}{\cdots \sum_{N_1=q}^{N_2}{a_{N_m}\cdots a_{N_1}}}
-\sum_{N_m=q}^{n}{\cdots \sum_{N_1=q}^{N_2}{a_{N_m}\cdots a_{N_1}}}
=a_{n+1}\left\{a_{n+1} \left[ \cdots a_{n+1} \left( a_{n+1}(1)+\sum_{N_1=q}^{n}{a_{N_1}} \right)+\sum_{N_2=q}^{n}{ \sum_{N_1=q}^{N_2}{a_{N_2} a_{N_1}}}\right] +\sum_{N_{m-1}=q}^{n}{\cdots \sum_{N_1=q}^{N_2}{a_{N_{m-1}}\cdots a_{N_1}}} \right\}.
\end{dmath*}
Using the notation from Eq.~\eqref{eq2}, this theorem can be written as 
\begin{dmath*}
\hat{R}_{m,q,n+1}=a_{n+1}\left\{ a_{n+1}\left[ \cdots a_{n+1} \left( a_{n+1} \left( \hat{R}_{0,q,n} \right) + \hat{R}_{1,q,n} \right) + \hat{R}_{2,q,n} \right] + \hat{R}_{m-1,q,n} \right\} + \hat{R}_{m,q,n}
\end{dmath*}
where $\hat R_{0,q,n}=1$.
\end{corollary}
\begin{example}
{\em Consider that $m=2$, we have the two following cases: }
\begin{itemize}
\item {\em If all sequences are distinct,  }
$$
\sum_{N_2=q}^{n+1}{b_{N_2} \sum_{N_1=q}^{N_2}{a_{N_1}}}
-\sum_{N_2=q}^{n}{b_{N_2} \sum_{N_1=q}^{N_2}{a_{N_1}}}
=(b_{n+1}) \left\{a_{n+1}(1)+ \sum_{N_1=q}^{n}{a_{N_1}} \right\}.
$$
\item {\em If all sequences are the same,  }
$$
\sum_{N_2=q}^{n+1}{a_{N_2} \sum_{N_1=q}^{N_2}{a_{N_1}}}
-\sum_{N_2=q}^{n}{a_{N_2} \sum_{N_1=q}^{N_2}{a_{N_1}}}
=(a_{n+1}) \left\{a_{n+1}(1)+\sum_{N_1=q}^{n}{a_{N_1}} \right\}.
$$
\end{itemize}
\end{example}
\subsection{General expression}
The variation of a recurrent sum can also be expressed in terms of only a certain range of lower order recurrent sums. In other words, $R_{m,q,n+1}$ can be expressed in terms of $R_{m,q,n}$ and of recurrent sums of order going only from $p$ to $(m-1)$. To do so, we develop the following theorem.  
\begin{theorem} \label{Theorem 2.3}
For any $m,q,n \in \mathbb{N}$ where $n \geq q$, for any $p \in [0,m]$, and for any set of sequences $a_{(1);N_1},\cdots,a_{(m);N_m}$ defined in the interval $[q,n+1]$, we have that 
\begin{dmath*}
\sum_{N_m=q}^{n+1}{\cdots \sum_{N_1=q}^{N_2}{a_{(m);N_m}\cdots a_{(1);N_1}}}
=\sum_{k=p+1}^{m}{\left(\prod_{j=0}^{m-k-1}{a_{(m-j);n+1}}\right)\left(\sum_{N_k=q}^{n}{\cdots \sum_{N_1=q}^{N_2}{a_{(k);N_k}\cdots a_{(1);N_1}}}\right)}
+\left(\prod_{j=0}^{m-p-1}{a_{(m-j);n+1}}\right)\left(\sum_{N_p=q}^{n+1}{\cdots \sum_{N_1=q}^{N_2}{a_{(p);N_p}\cdots a_{(1);N_1}}}\right).
\end{dmath*}
Using the notation from Eq.~\eqref{eq1}, this theorem can be written as 
$$
R_{m,q,n+1}
=\sum_{k=p+1}^{m}{\left(\prod_{j=0}^{m-k-1}{a_{(m-j);n+1}}\right)R_{k,q,n}}
+\left(\prod_{j=0}^{m-p-1}{a_{(m-j);n+1}}\right)R_{p,q,n+1}.
$$
\end{theorem}
\begin{proof}
By applying Theorem \ref{Theorem 2.1}, 
\begin{equation*}
\begin{split}
R_{m,q,n+1}
&=\sum_{k=0}^{m}{\left( \prod_{j=0}^{m-k-1}{a_{(m-j);n+1}} \right)R_{k,q,n}}\\
&=\sum_{k=p+1}^{m}{\left( \prod_{j=0}^{m-k-1}{a_{(m-j);n+1}} \right)R_{k,q,n}}
+\sum_{k=0}^{p}{\left( \prod_{j=0}^{m-k-1}{a_{(m-j);n+1}} \right)R_{k,q,n}}\\
&=\sum_{k=p+1}^{m}{\left( \prod_{j=0}^{m-k-1}{a_{(m-j);n+1}} \right)R_{k,q,n}}
+\left( \prod_{j=0}^{m-p-1}{a_{(m-j);n+1}} \right) \sum_{k=0}^{p}{\left( \prod_{j=m-p}^{m-k-1}{a_{(m-j);n+1}} \right)R_{k,q,n}}.
\end{split}
\end{equation*} 
From Theorem \ref{Theorem 2.1}, with $m$ substituted by $p$, we have  
$$
R_{p,q,n+1}
=\sum_{k=0}^{p}{\left(\prod_{j=0}^{p-k-1}{a_{(p-j);n+1}}\right)R_{k,q,n}}
=\sum_{k=0}^{p}{\left(\prod_{j=m-p}^{m-k-1}{a_{(m-j);n+1}}\right)R_{k,q,n}}
.$$
Hence, by substituting, we get 
$$
R_{m,q,n+1}
=\sum_{k=p+1}^{m}{\left( \prod_{j=0}^{m-k-1}{a_{(m-j);n+1}} \right)R_{k,q,n}}
+\left( \prod_{j=0}^{m-p-1}{a_{(m-j);n+1}} \right)R_{p,q,n+1}
.$$
\end{proof}
\begin{corollary} \label{Corollary 2.4}
If all sequences are the same, Theorem \ref{Theorem 2.3} will be reduced to the following form,  
\begin{dmath*}
\sum_{N_m=q}^{n+1}{\cdots \sum_{N_1=q}^{N_2}{a_{N_m}\cdots a_{N_1}}}
=\sum_{k=p+1}^{m}{\left(a_{n+1}\right)^{m-k}\left(\sum_{N_k=q}^{n}{\cdots \sum_{N_1=q}^{N_2}{a_{N_k}\cdots a_{N_1}}}\right)}
+\left(a_{n+1}\right)^{m-p}\left(\sum_{N_p=q}^{n+1}{\cdots \sum_{N_1=q}^{N_2}{a_{N_p}\cdots a_{N_1}}}\right).
\end{dmath*}
Using the notation from Eq.~\eqref{eq2}, this theorem can be written as 
$$
\hat{R}_{m,q,n+1}
=\sum_{k=p+1}^{m}{\left(a_{n+1}\right)^{m-k}\hat{R}_{k,q,n}}
+\left(a_{n+1}\right)^{m-p}\hat{R}_{p,q,n+1}.
$$
\end{corollary}
\begin{example}
{\em For $p=2$ and if the sequences are the same: }
\begin{dmath*}
\sum_{N_m=q}^{n+1}{\cdots \sum_{N_1=q}^{N_2}{a_{N_m}\cdots a_{N_1}}}
=\sum_{k=3}^{m}{\left(a_{n+1}\right)^{m-k}\left(\sum_{N_k=q}^{n}{\cdots \sum_{N_1=q}^{N_2}{a_{N_k}\cdots a_{N_1}}}\right)}
+\left(a_{n+1}\right)^{m-2}\left(\sum_{N_2=q}^{n+1}{\sum_{N_1=q}^{N_2}{a_{N_2}a_{N_1}}}\right).
\end{dmath*}
\end{example}
\begin{example}
{\em For $p=m-2$ and if the sequences are the same: }
\begin{dmath*}
\sum_{N_m=q}^{n+1}{\cdots \sum_{N_1=q}^{N_2}{a_{N_m}\cdots a_{N_1}}}
=\left(\sum_{N_m=q}^{n}{\cdots \sum_{N_1=q}^{N_2}{a_{N_m}\cdots a_{N_1}}}\right)
+\left(a_{n+1}\right)\left(\sum_{N_{m-1}=q}^{n}{\cdots \sum_{N_1=q}^{N_2}{a_{N_{m-1}}\cdots a_{N_1}}}\right)
+\left(a_{n+1}\right)^{2}\left(\sum_{N_{m-2}=q}^{n+1}{\cdots \sum_{N_1=q}^{N_2}{a_{N_{m-2}}\cdots a_{N_1}}}\right).
\end{dmath*}
\end{example}
\begin{remark} \label{Corollary 2.5}
{\em Set $a_{(m);N}=\cdots = a_{(2);N}=1$, Theorem \ref{Theorem 2.3} becomes }
$$
\sum_{N_m=q}^{n+1}{\sum_{N_{m-1}=q}^{N_m}{\cdots \sum_{N_1=q}^{N_2}{a_{N_1}}}}
=\sum_{k=p+1}^{m}{\left( \sum_{N_k=q}^{n}{\sum_{N_{k-1}=q}^{N_k}{\cdots \sum_{N_1=q}^{N_2}{a_{N_1}}}}\right)}
+\sum_{N_p=q}^{n+1}{\sum_{N_{p-1}=q}^{N_p}{\cdots \sum_{N_1=q}^{N_2}{a_{N_1}}}}
.$$
\end{remark}
\subsection{General recurrent expression}
Similarly, the theorem introduced in the previous section can be reformulated in a recursive form by expanding and factoring the expression of Theorem \ref{Theorem 2.3} to obtain the following expression.  
\begin{theorem} \label{Theorem 2.4}
For any $m,q,n \in \mathbb{N}$ where $n \geq q$, for any $p \in [0,m]$, and for any set of sequences $a_{(1);N_1},\cdots,a_{(m);N_m}$ defined in the interval $[q,n+1]$, we have that 
\begin{dmath*}
\Scale[0.87]{
R_{m,q,n+1}=a_{(m);n+1}\left\{ a_{(m-1);n+1}\left[ \cdots a_{(p+2);n+1} \left( a_{(p+1);n+1} \left( R_{p,q,n+1} \right) + R_{p+1,q,n} \right) + R_{p+2,q,n} \right] + R_{m-1,q,n} \right\} + R_{m,q,n}. }
\end{dmath*}
\end{theorem}
\begin{proof}
From Theorem \ref{Theorem 2.2}, with $m$ substituted by $p$, we have 
\begin{dmath*}
R_{p,q,n+1}=a_{(p);n+1}\left\{ a_{(p-1);n+1}\left[ \cdots a_{(2);n+1} \left( a_{(1);n+1} \left( R_{0,q,n} \right) + R_{1,q,n} \right) + R_{2,q,n} \right] + R_{p-1,q,n} \right\} + R_{p,q,n}
\end{dmath*}
where $R_{0,q,n}=1$. \\
Substituting into the expression of Theorem \ref{Theorem 2.2}, the inner part becomes $R_{p,q,n+1}$ and we get the desired formula.  
\end{proof}
\begin{corollary} \label{Corollary 2.6}
If all sequences are the same, Theorem \ref{Theorem 2.4} will be reduced to the following form, \begin{dmath*}
\hat{R}_{m,q,n+1}=a_{n+1}\left\{ a_{n+1}\left[ \cdots a_{n+1} \left( a_{n+1} \left( \hat{R}_{p,q,n+1} \right) + \hat{R}_{p+1,q,n} \right) + \hat{R}_{p+2,q,n} \right] + \hat{R}_{m-1,q,n} \right\} + \hat{R}_{m,q,n}.
\end{dmath*}
\end{corollary}
\begin{example}
{\em For $p=1$ and if the sequences are the same: }
\begin{dmath*}
\sum_{N_m=q}^{n+1}{\cdots \sum_{N_1=q}^{N_2}{a_{N_m}\cdots a_{N_1}}}
-\sum_{N_m=q}^{n}{\cdots \sum_{N_1=q}^{N_2}{a_{N_m}\cdots a_{N_1}}}
=a_{n+1}\left\{a_{n+1} \left[ \cdots a_{n+1} \left(\sum_{N_1=q}^{n+1}{a_{N_1}} \right)+\sum_{N_2=q}^{n}{ \sum_{N_1=q}^{N_2}{a_{N_2} a_{N_1}}}\right] +\sum_{N_{m-1}=q}^{n}{\cdots \sum_{N_1=q}^{N_2}{a_{N_{m-1}}\cdots a_{N_1}}} \right\}.
\end{dmath*}
\end{example}
\begin{example}
{\em For $p=m-2$ and if the sequences are the same: }
\begin{dmath*}
\sum_{N_m=q}^{n+1}{\cdots \sum_{N_1=q}^{N_2}{a_{N_m}\cdots a_{N_1}}}
-\sum_{N_m=q}^{n}{\cdots \sum_{N_1=q}^{N_2}{a_{N_m}\cdots a_{N_1}}}
=a_{n+1}\left\{a_{n+1} \left[ \sum_{N_{m-2}=q}^{n+1}{\cdots \sum_{N_1=q}^{N_2}{a_{N_{m-2}}\cdots a_{N_1}}}\right] +\sum_{N_{m-1}=q}^{n}{\cdots \sum_{N_1=q}^{N_2}{a_{N_{m-1}}\cdots a_{N_1}}} \right\}.
\end{dmath*}
\end{example}
\section{Inversion Formulas} \label{Inversion Formulas}
In this section, we will develop formulas to interchange the order of summation in a recurrent sum. 
\subsection{Particular case (for 2 sequences)}
We start by proving the inversion formula with $2$ sequences which is required in order to prove the more general inversion formula with $m$ sequences. 
\begin{theorem} \label{Theorem 3.1}
For $m,q,n \in \mathbb{N}$ where $n \geq q$ and for any 2 sequences $a_{N_1}$ and $b_{N_2}$ defined in the interval $[q,n]$, we have that 
$$
\sum_{N_2=q}^{n}{b_{N_2}\sum_{N_1=q}^{N_2}{a_{N_1}}}
=\sum_{N_1=q}^{n}{a_{N_1}\sum_{N_2=N_1}^{n}{b_{N_2}}}
.$$
\end{theorem}
\begin{proof}
By expanding the sum, we get 
\begin{dmath*}
\sum_{N_2=q}^{n}{b_{N_2}\sum_{N_1=q}^{N_2}{a_{N_1}}}
=b_q \left( \sum_{N_1=q}^{q}{a_{N_1}}\right)
+b_{q+1} \left( \sum_{N_1=q}^{q+1}{a_{N_1}}\right)
+\cdots 
+b_{n-1} \left( \sum_{N_1=q}^{n-1}{a_{N_1}}\right)
+b_{n} \left( \sum_{N_1=q}^{n}{a_{N_1}}\right)
=b_q \left( a_{q}\right)
+b_{q+1} \left( a_{q}+a_{q+1}\right)
+\cdots 
+b_{n-1} \left( a_{q}+ \cdots + a_{n-1}\right)
+b_{n} \left( a_{q}+ \cdots + a_{n}\right)
.\end{dmath*}
By regrouping the $b_N$ terms instead of the $a_N$ terms, the expression becomes 
\begin{dmath*}
\sum_{N_2=q}^{n}{b_{N_2}\sum_{N_1=q}^{N_2}{a_{N_1}}}
=a_q \left( b_{q}+\cdots + b_{n}\right)
+a_{q+1} \left( b_{q+1}+\cdots + b_{n}\right)
+\cdots 
+a_{n-1} \left(b_{n-1}+b_{n}\right)
+a_{n} \left(b_{n}\right)
=a_q \left( \sum_{N_2=q}^{n}{b_{N_2}}\right)
+a_{q+1} \left( \sum_{N_2=q+1}^{n}{b_{N_2}}\right)
+\cdots 
+a_{n-1} \left( \sum_{N_2=n-1}^{n}{b_{N_2}}\right)
+a_{n} \left( \sum_{N_2=n}^{n}{b_{N_2}}\right)
=\sum_{N_1=q}^{n}{a_{N_1}\sum_{N_2=N_1}^{n}{b_{N_2}}}. 
\end{dmath*}
\end{proof}
\begin{corollary} \label{Corollary 3.1}
If all sequences are the same, Theorem \ref{Theorem 3.1} becomes 
$$
\sum_{N_2=q}^{n}{a_{N_2}\sum_{N_1=q}^{N_2}{a_{N_1}}}
=\sum_{N_1=q}^{n}{a_{N_1}\sum_{N_2=N_1}^{n}{a_{N_2}}}
.$$
\end{corollary}
\subsection{General case (for m sequences)}
We now prove the more general inversion formula with $m$ sequences which allows us to invert the order of summation for a recurrent sum of order $m$.  
\begin{theorem} \label{Theorem 3.2}
For any $m,q,n \in \mathbb{N}$ where $n \geq q$ and for any set of sequences $a_{(1);N_1},\cdots,a_{(m);N_m}$ defined in the interval $[q,n]$, we have that 
$$
\sum_{N_m=q}^{n}{a_{(m);N_m}\cdots\sum_{N_2=q}^{N_{3}}{a_{(2);N_2}\sum_{N_1=q}^{N_2}{a_{(1);N_1}}}}
=\sum_{N_1=q}^{n}{a_{(1);N_1}\sum_{N_2=N_1}^{n}{a_{(2);N_2}\cdots\sum_{N_m=N_{m-1}}^{n}{a_{(m);N_m}}}}
.$$
\end{theorem}
\begin{proof}
1. Base Case: verify true for $m=2$. \\
This statement is true as proven in Theorem \ref{Theorem 3.1}. \\
2. Induction hypothesis: assume the statement is true until $m$. 
$$
\sum_{N_m=q}^{n}{a_{(m);N_m}\cdots\sum_{N_2=q}^{N_{3}}{a_{(2);N_2}\sum_{N_1=q}^{N_2}{a_{(1);N_1}}}}
=\sum_{N_1=q}^{n}{a_{(1);N_1}\sum_{N_2=N_1}^{n}{a_{(2);N_2}\cdots\sum_{N_m=N_{m-1}}^{n}{a_{(m);N_m}}}}
.$$
3. Induction step: we will show that this statement is true for $(m+1)$. \\
We have to show the following statement to be true: 
$$
\sum_{N_{m+1}=q}^{n}{a_{(m+1);N_{m+1}}\cdots\sum_{N_2=q}^{N_{3}}{a_{(2);N_2}\sum_{N_1=q}^{N_2}{a_{(1);N_1}}}}
=\sum_{N_1=q}^{n}{a_{(1);N_1}\sum_{N_2=N_1}^{n}{a_{(2);N_2}\cdots\sum_{N_{m+1}=N_{m}}^{n}{a_{(m+1);N_{m+1}}}}}
.$$
$ \\ $
\begin{dmath*}
\sum_{N_{m+1}=q}^{n}{a_{(m+1);N_{m+1}}\cdots\sum_{N_2=q}^{N_{3}}{a_{(2);N_2}\sum_{N_1=q}^{N_2}{a_{(1);N_1}}}}
=\sum_{N_{m+1}=q}^{n}{a_{(m+1);N_{m+1}}\left(\sum_{N_m=q}^{N_{m+1}}{a_{(m);N_m}\cdots\sum_{N_2=q}^{N_{3}}{a_{(2);N_2}\sum_{N_1=q}^{N_2}{a_{(1);N_1}}}}\right)}
.\end{dmath*} 
Let $b_{N_m}$ be the following sequence (that dependents only on $N_m$), 
$$
b_{N_m}=a_{(m);N_m}\sum_{N_{m-1}=q}^{N_{m}}{a_{(m-1);N_{m-1}}\cdots\sum_{N_2=q}^{N_{3}}{a_{(2);N_2}\sum_{N_1=q}^{N_2}{a_{(1);N_1}}}}
.$$ 
By applying this substitution in the previous expression, we obtain a recurrent sum of order 2 that contains the 2 sequences $a_{(m+1);N_{m+1}}$ and $b_{N_m}$. Then, we apply the inversion formula for the case of 2 sequences (Theorem \ref{Theorem 3.1}) to get the following, 
\begin{dmath*}
\sum_{N_{m+1}=q}^{n}{a_{(m+1);N_{m+1}}\cdots\sum_{N_2=q}^{N_{3}}{a_{(2);N_2}\sum_{N_1=q}^{N_2}{a_{(1);N_1}}}}
=\sum_{N_{m+1}=q}^{n}{a_{(m+1);N_{m+1}}\left(\sum_{N_m=q}^{N_{m+1}}{b_{N_m}}\right)}
=\sum_{N_m=q}^{n}{b_{N_m}\left(\sum_{N_{m+1}=N_m}^{n}{a_{(m+1);N_{m+1}}}\right)}
=\sum_{N_m=q}^{n}{a_{(m);N_m}\cdots\sum_{N_2=q}^{N_{3}}{a_{(2);N_2}\sum_{N_1=q}^{N_2}{a_{(1);N_1}\left(\sum_{N_{m+1}=N_m}^{n}{a_{(m+1);N_{m+1}}}\right)}}}
.\end{dmath*} 
The sum of $a_{(m+1);N_{m+1}}$ has $N_m$ and $n$ as lower and upper bounds. Thus, knowing that $n$ is a constant, the sum of $a_{(m+1);N_{m+1}}$ depends only on $N_m$. This allows us to  extract this sum from the inner sums to get 
\begin{dmath*}
\sum_{N_{m+1}=q}^{n}{a_{(m+1);N_{m+1}}\cdots\sum_{N_2=q}^{N_{3}}{a_{(2);N_2}\sum_{N_1=q}^{N_2}{a_{(1);N_1}}}}
=\sum_{N_m=q}^{n}{\left(a_{(m);N_m}\sum_{N_{m+1}=N_m}^{n}{a_{(m+1);N_{m+1}}}\right)\cdots\sum_{N_2=q}^{N_{3}}{a_{(2);N_2}\sum_{N_1=q}^{N_2}{a_{(1);N_1}}}}
.\end{dmath*} 
Let $A_{N_m}$ be the following sequence (that only depends on $N_m$), 
$$
A_{N_m}
=a_{(m);N_m}\sum_{N_{m+1}=N_m}^{n}{a_{(m+1);N_{m+1}}}
.$$
By substituting $A_{N_m}$ into the previous expression, we get a recurrent sum of order $m$ in terms of the following $m$ sequences:  $A_{N_m}, a_{(m-1);N_{m-1}}, \cdots , a_{(1);N_1}$. Then the inversion formula for the case of $m$ sequences (which was assumed to be true in the induction hypothesis) is applied, 
\begin{dmath*}
\sum_{N_{m+1}=q}^{n}{a_{(m+1);N_{m+1}}\cdots\sum_{N_2=q}^{N_{3}}{a_{(2);N_2}\sum_{N_1=q}^{N_2}{a_{(1);N_1}}}}
=\sum_{N_m=q}^{n}{A_{N_m}\cdots\sum_{N_2=q}^{N_{3}}{a_{(2);N_2}\sum_{N_1=q}^{N_2}{a_{(1);N_1}}}}
=\sum_{N_1=q}^{n}{a_{(1);N_1}\sum_{N_2=N_1}^{n}{a_{(2);N_2}\cdots\sum_{N_m=N_{m-1}}^{n}{A_{N_m}}}}
=\sum_{N_1=q}^{n}{a_{(1);N_1}\sum_{N_2=N_1}^{n}{a_{(2);N_2}\cdots\sum_{N_m=N_{m-1}}^{n}{a_{(m);N_m}\sum_{N_{m+1}=N_m}^{n}{a_{(m+1);N_{m+1}}}}}}
.\end{dmath*}
We conclude that it must hold for all $m \geq 2$. 
\end{proof}
\begin{corollary} \label{Corollary 3.2}
If all sequences are the same, Theorem \ref{Theorem 3.2} becomes 
$$
\sum_{N_m=q}^{n}{a_{N_m}\cdots\sum_{N_2=q}^{N_{3}}{a_{N_2}\sum_{N_1=q}^{N_2}{a_{N_1}}}}
=\sum_{N_1=q}^{n}{a_{N_1}\sum_{N_2=N_1}^{n}{a_{N_2}\cdots\sum_{N_m=N_{m-1}}^{n}{a_{N_m}}}}
.$$
\end{corollary}
Similarly, the innermost summation can be turned into the outermost summation as illustrated by Theorem \ref{Theorem 3.3}.  
\begin{theorem} \label{Theorem 3.3}
For any $m,q,n \in \mathbb{N}$ where $n \geq q$ and for any set of sequences $a_{(1);N_1},\cdots,a_{(m);N_m}$ defined in the interval $[q,n]$, we have that 
$$
\sum_{N_m=q}^{n}{a_{(m);N_m}\cdots\sum_{N_2=q}^{N_{3}}{a_{(2);N_2}\sum_{N_1=q}^{N_2}{a_{(1);N_1}}}}
=\sum_{N_1=q}^{n}{a_{(1);N_1}\sum_{N_m=N_1}^{n}{a_{(m);N_m}\cdots\sum_{N_3=N_{1}}^{N_4}{a_{(3);N_3}\sum_{N_2=N_{1}}^{N_3}{a_{(2);N_2}}}}}
.$$
\end{theorem}
\begin{proof}
From Theorem \ref{Theorem 3.2}, 
$$
\sum_{N_m=q}^{n}{a_{(m);N_m}\cdots\sum_{N_2=q}^{N_{3}}{a_{(2);N_2}\sum_{N_1=q}^{N_2}{a_{(1);N_1}}}}
=\sum_{N_1=q}^{n}{a_{(1);N_1}\sum_{N_2=N_1}^{n}{a_{(2);N_2}\cdots\sum_{N_m=N_{m-1}}^{n}{a_{(m);N_m}}}}
.$$
Applying Theorem \ref{Theorem 3.2} to the inner part of the right side sum would transform it as follows 
$$
\sum_{N_2=N_1}^{n}{a_{(2);N_2}\cdots\sum_{N_m=N_{m-1}}^{n}{a_{(m);N_m}}}
=\sum_{N_m=N_1}^{n}{a_{(m);N_m}\cdots\sum_{N_3=N_{1}}^{N_4}{a_{(3);N_3}\sum_{N_2=N_{1}}^{N_3}{a_{(2);N_2}}}}
.$$
Hence, substituting back into Theorem \ref{Theorem 3.2} would give us the desired formula. 
\end{proof}
\begin{corollary} \label{Corollary 3.3}
If all sequences are the same, Theorem \ref{Theorem 3.3} becomes  
$$
\sum_{N_m=q}^{n}{a_{N_m}\cdots\sum_{N_2=q}^{N_{3}}{a_{N_2}\sum_{N_1=q}^{N_2}{a_{N_1}}}}
=\sum_{N_1=q}^{n}{a_{N_1}\sum_{N_m=N_1}^{n}{a_{N_m}\cdots\sum_{N_3=N_{1}}^{N_4}{a_{N_3}\sum_{N_2=N_{1}}^{N_3}{a_{N_2}}}}}
.$$
\end{corollary}
\subsection{Inversion of p sequences from m sequences}
Finally, as we will show in this section, it is possible to partially invert the order of summation for a recurrent sum. In other words, as shown by the following theorem, it is possible to invert the order of summation of only the $p$ innermost summations from $m$ summations. 
\begin{theorem} \label{Theorem 3.4}
For any $m,q,n \in \mathbb{N}$ where $n \geq q$, for any $p \in [0,m]$, and for any set of sequences $a_{(1);N_1},\cdots,a_{(m);N_m}$ defined in the interval $[q,n]$, we have that 
\begin{dmath*}
\sum_{N_m=q}^{n}{a_{(m);N_m}\cdots\sum_{N_p=q}^{N_{p+1}}{a_{(p);N_p}\cdots\sum_{N_1=q}^{N_2}{a_{(1);N_1}}}}
=\sum_{N_m=q}^{n}{a_{(m);N_m}\cdots\sum_{N_{p+1}=q}^{N_{p+2}}{a_{(p+1);N_{p+1}}\sum_{N_1=q}^{N_{p+1}}{a_{(1);N_1}\sum_{N_2=N_1}^{N_{p+1}}{a_{(2);N_2}\cdots\sum_{N_p=N_{p-1}}^{N_{p+1}}{a_{(p);N_p}}}}}}.
\end{dmath*}
\end{theorem}
\begin{proof}
By replacing $m$ by $p$ and $n$ by $N_{p+1}$ in Theorem \ref{Theorem 3.2}, we get the following relation, 
$$
\sum_{N_p=q}^{N_{p+1}}{a_{(p);N_p}\cdots\sum_{N_2=q}^{N_{3}}{a_{(2);N_2}\sum_{N_1=q}^{N_2}{a_{(1);N_1}}}}
=\sum_{N_1=q}^{N_{p+1}}{a_{(1);N_1}\sum_{N_2=N_1}^{N_{p+1}}{a_{(2);N_2}\cdots\sum_{N_p=N_{p-1}}^{N_{p+1}}{a_{(p);N_p}}}}
.$$
Thus, 
\begin{dmath*}
\sum_{N_m=q}^{n}{a_{(m);N_m}\cdots\sum_{N_p=q}^{N_{p+1}}{a_{(p);N_p}\cdots\sum_{N_1=q}^{N_2}{a_{(1);N_1}}}}
=\sum_{N_m=q}^{n}{a_{(m);N_m}\cdots \sum_{N_{p+1}}^{N_{p+2}}{a_{(p+1);N_{p+1}}\left(\sum_{N_p=q}^{N_{p+1}}{a_{(p);N_p}\cdots\sum_{N_1=q}^{N_2}{a_{(1);N_1}}}\right)}}
=\sum_{N_m=q}^{n}{a_{(m);N_m}\cdots\sum_{N_{p+1}=q}^{N_{p+2}}{a_{(p+1);N_{p+1}}\left(\sum_{N_1=q}^{N_{p+1}}{a_{(1);N_1}\sum_{N_2=N_1}^{N_{p+1}}{a_{(2);N_2}\cdots\sum_{N_p=N_{p-1}}^{N_{p+1}}{a_{(p);N_p}}}}\right)}}
.\end{dmath*}
\end{proof}
\begin{corollary} \label{Corollary 3.4}
If all sequences are the same, Theorem \ref{Theorem 3.4} becomes   
\begin{dmath*}
\sum_{N_m=q}^{n}{a_{N_m}\cdots\sum_{N_p=q}^{N_{p+1}}{a_{N_p}\cdots\sum_{N_1=q}^{N_2}{a_{N_1}}}}
=\sum_{N_m=q}^{n}{a_{N_m}\cdots\sum_{N_{p+1}=q}^{N_{p+2}}{a_{N_{p+1}}\sum_{N_1=q}^{N_{p+1}}{a_{N_1}\sum_{N_2=N_1}^{N_{p+1}}{a_{N_2}\cdots\sum_{N_p=N_{p-1}}^{N_{p+1}}{a_{N_p}}}}}}.
\end{dmath*}
\end{corollary}
Similarly, the innermost summation can be pulled back to the $p$-th position as illustrated by Theorem \ref{Theorem 3.5}.   
\begin{theorem} \label{Theorem 3.5}
For any $m,q,n \in \mathbb{N}$ where $n \geq q$, for any $p \in [0,m]$, and for any set of sequences $a_{(1);N_1},\cdots,a_{(m);N_m}$ defined in the interval $[q,n]$, we have that 
\begin{dmath*}
\sum_{N_m=q}^{n}{a_{(m);N_m}\cdots\sum_{N_p=q}^{N_{p+1}}{a_{(p);N_p}\cdots\sum_{N_1=q}^{N_2}{a_{(1);N_1}}}}
=\sum_{N_m=q}^{n}{a_{(m);N_m}\cdots\sum_{N_{p+1}=q}^{N_{p+2}}{a_{(p+1);N_{p+1}}\sum_{N_1=q}^{N_{p+1}}{a_{(1);N_1}\sum_{N_p=N_1}^{N_{p+1}}{a_{(p);N_p}\sum_{N_{p-1}=N_1}^{N_{p}}{a_{(p-1);N_{p-1}}\cdots\sum_{N_2=N_{1}}^{N_{3}}{a_{(2);N_2}}}}}}}.
\end{dmath*}
\end{theorem}
\begin{proof}
By applying Theorem \ref{Theorem 3.3} (with $m$ substituted by $p$ and $n$ substituted by $N_{p+1}$) to Theorem \ref{Theorem 3.4}, we get the desired theorem. 
\end{proof}
\begin{corollary}  \label{Corollary 3.5}
If all sequences are the same, Theorem \ref{Theorem 3.5} becomes 
\begin{dmath*}
\sum_{N_m=q}^{n}{a_{N_m}\cdots\sum_{N_p=q}^{N_{p+1}}{a_{N_p}\cdots\sum_{N_1=q}^{N_2}{a_{N_1}}}}
=\sum_{N_m=q}^{n}{a_{N_m}\cdots\sum_{N_{p+1}=q}^{N_{p+2}}{a_{N_{p+1}}\sum_{N_1=q}^{N_{p+1}}{a_{N_1}\sum_{N_p=N_1}^{N_{p+1}}{a_{N_p}\sum_{N_{p-1}=N_1}^{N_{p}}{a_{N_{p-1}}\cdots\sum_{N_2=N_{1}}^{N_{3}}{a_{N_2}}}}}}}.
\end{dmath*}
\end{corollary}
\section{Reduction Formulas} \label{Reduction Formulas}
The objective of this section is to introduce formulas which can be used to reduce recurrent sums from their originally recurrent form $\left(\sum_{N_m=q}^{n}{\cdots \sum_{N_1=q}^{N_2}{a_{N_m} \cdots a_{N_1}}}\right)$ to a form containing only simple non-recurrent sums $\left(\left(\sum_{N=q}^{n}{(a_N)^i}\right)^y\right)$. 
\subsection{A brief introduction to partitions}
In this paper, partitions are involved in the reduction formula for a recurrent sum. For this reason, in this section, we will present a brief introduction to partitions.  
\begin{definition}
A partition of a non-negative integer $m$ is a set of positive integers whose sum equals $m$.
We can represent a partition of $m$ as a vector $(y_{k,1},\cdots,y_{k,m})$ that verifies
\begin{align}
\begin{pmatrix}
y_{k,1} \\
\vdots \\
y_{k,m} \\
\end{pmatrix}
\cdot
\begin{pmatrix}
1 \\
\vdots \\
m \\
\end{pmatrix}
=y_{k,1}+2y_{k,2}+ \cdots + my_{k,m}=m
.\end{align}
\end{definition}
The set of partitions of a non-negative integer $m$ is the set of vectors $(y_{k,1},\cdots,y_{k,m})$ that verify the previous identity. We will denote this set by $P$. The cardinality of this set is equal to the number of partitions of $m$ (which is the partition function denoted by $p(m)$),
\begin{equation}
\text{Card}(P)=p(m).
\end{equation}
Hence, the set of partitions of $m$ is the set of vectors $\{(y_{1,1},\cdots,y_{1,m}),(y_{2,1},\cdots,y_{2,m}), \cdots \}$ which consists of $p(m)$ vectors. The value of $p(m)$ is obtained from the generating function developed by Euler in the mid-eighteen century \cite{PartitionEuler}, 
\begin{equation}
\sum_{m=0}^{\infty}{p(m)x^m}=\prod_{j=1}^{\infty}{\frac{1}{1-x^j}}
.\end{equation}
Euler also showed that this relation implies the following recurrent definition for $p(m)$, 
\begin{equation}
p(m)=\sum_{j=1}^{\infty}{(-1)^{j-1} \left(p \left(m- \frac{j(3j-1)}{2} \right)-p \left(m- \frac{j(3j+1)}{2} \right) \right)}.
\end{equation}
In 1918, Hardy and Ramanujan provided an asymptotic expression for $p(m)$ in \cite{Hardy}. Later, in 1937, Rademacher was able to improve on Hardy and Ramanujan’s formula by proving the following expression for $p(m)$ in \cite{Rademacher}, 
\begin{equation}
p(m)= \frac{1}{\pi \sqrt{2}} \sum_{k=1}^{\infty}{\sqrt{k} A_{k}(m) \frac{d}{dm} \left[ \frac{\sinh \left( \frac{\pi}{k} \sqrt{\frac{2}{3} \left(m- \frac{1}{24} \right)} \right)}{\sqrt{m- \frac{1}{24}}} \right]}
\end{equation}
where $A_{k}(m)$ is a Kloosterman type sum, 
\begin{equation}
A_{k}(m)=\sum_{\substack{0 \leq h <k \\ gcd(h,k)=1}}{e^{\pi i (s(h,k)-2mh/k)}}
\end{equation}
and where the notation $s(m,k)$ represents a Dedekind sum. \\
However, this formula has the disadvantage of being an infinite sum. This formula remained the only exact explicit formula for $p(m)$ until Ono and Bruinier presented a new formula for $p(m)$ as a finite sum \cite{Ono}.
 \\ 
Additionally, two of the most famous ways of representing a partition are using Ferrers diagrams or using Young diagrams. Similarly, there exists some variants of Ferrers diagrams that are used (see \cite{propp1989some}). 
\begin{remark}
{\em For readers intrested in a more detailed explanation of partition, see \cite{andrews1998theory}. }
\end{remark}
\subsection{Reduction Theorem and Partition Identities}
We will start this section by proving several lemmas which are needed in order to prove the main theorem of this section (Theorem \ref{Theorem 4.1}, which we will call the reduction theorem). However, some of these lemmas are important on their own as they provide relations governing partitions. \\

We start by proving the following trivial lemma.  
\begin{lemma} \label{Lemma 4.1}
No partition of a non-negative integer $m$ constructed from a sum of $r$ terms (positive integers) can contain an integer larger or equal to $m-r+2$. 
\end{lemma}
\begin{proof}
The smallest sum of $r$ positive integers containing $i$ is $i+\underbrace{(1+ \cdots +1)}_{(r-1)}=i+(r-1)$. \\
If $i \geq m-r+2$ then $i+(r-1)\geq m-r+2+r-1=m+1>m$. \\
Hence, such a sum, being strictly larger than $m$, cannot be a partition of $m$. 
\end{proof}
Before we can proceed to prove the other needed lemmas, we need to define the following notation: Let $[x^r]\left(P(x)\right)$ represent the coefficient of $x^r$ in $P(x)$. Let $x^{\overline{m}}=x(x+1)\cdots (x+m-1)$ represent the Rising factorial. Let $(x)_m=x(x-1)\cdots (x-m+1)$ represent the Falling factorial. \\
The original definition of Stirling numbers of the first kind $S(m,r)$ was as the coefficients in the expansion of $(x)_m$:
\begin{equation}
(x)_m
=\sum_{k=0}^{m}{S(m,k)x^k} \,\,\,\,\,\,\,\, or \,\,\,\,\,\,\,\,   S(m,r)=[x^r](x)_m
.\end{equation}
In a similar way, the unsigned Stirling numbers of the first kind, denoted $|S(m,r)|$ or ${m \brack r}$, can be expressed in terms of the Rising factorial $x^{\overline{m}}$:
\begin{equation}
x^{\overline{m}}
=\sum_{k=0}^{m}{{m \brack k}x^k} \,\,\,\,\,\,\,\, or \,\,\,\,\,\,\,\,   {m \brack r}=[x^r]\left(x^{\overline{m}}\right)
.\end{equation}
From this definition, the famous finite sum of the unsigned Stirling numbers of the first kind can be directly deduced by substituting $x$ by 1 to get 
\begin{equation}
\sum_{k=0}^{m}{{m \brack k}}
=1(1+1)\cdots (1+m-1)=m! 
.\end{equation}
Note that $|S(m,r)|$ can also be defined as the number of permutations of $m$ elements with $r$ disjoint cycles. Similarly, the previous relation can be obtained by noticing that permutations are partitioned by number of cycles. 
\begin{remark}
{\em More details on Stirling numbers of the first kind can be found in  \cite{loeb1992generalization}. }
\end{remark}
For simplicity, we define $\sum{f(i)}$ to mean $\sum_{i=1}^{m}{f(i)}$. In particular, $\sum{i.y_{k,i}}=\sum_{i=1}^{m}{i.y_{k,i}}$ and $\sum{y_{k,i}}=\sum_{i=1}^{m}{y_{k,i}}$. 
Additionally, let a partition of $m$ of length $r$ refer to a partition $(y_{k,1},\cdots,y_{k,m})$ of $m$ such that $\sum{y_{k,i}}=r$. \\
Now that we have defined the needed notation, we can continue proving the required lemmas.  
\begin{lemma} \label{Lemma 4.2}
Let $m$ and $r$ be two non-negative integers with $r \leq m$, the following sum over partitions of $m$ of length $r$ can be expressed in terms of the unsigned Stirling numbers of the first kind as follows, 
$$
\sum_{\substack{k \\ \sum{i.y_{k,i}}=m \\ \sum{y_{k,i}}=r}}{\prod_{i=1}^{m}{\frac{1}{i^{y_{k,i}}(y_{k,i})!}}}
=\frac{1}{m!}{m \brack r}
.$$
\end{lemma}
\begin{proof}
A Bell polynomial is defined as follows 
\begin{dmath*}
B_{m,r}(x_1,x_2,\cdots , x_{m-r+1})
=\sum_{\substack{y_1+2y_2+\cdots +(m-r+1)y_{m-r+1}=m \\ y_1+y_2+\cdots +y_{m-r+1}=r}}{\frac{m!}{y_1!y_2! \cdots y_{m-r+1}!}\left(\frac{x_1}{1!}\right)^{y_1} \left(\frac{x_2}{2!}\right)^{y_2} \cdots \left(\frac{x_{m-r+1}}{(m-r+1)!}\right)^{y_{m-r+1}}}
.\end{dmath*}
These polynomials can also be rewritten more compactly as 
$$
B_{m,r}(x_1,x_2,\cdots , x_{m-r+1})
=m!\sum_{\substack{k \\ \sum{i.y_{k,i}}=m \\ \sum{y_{k,i}}=r}}{\prod_{i=1}^{m-r+1}{\frac{1}{y_{k,i}!} \left(\frac{x_i}{i!}\right)^{y_{k,i}}}}
.$$
A property of the Bell polynomial, shown in \cite{wang2009general}, is that the value of the Bell polynomial on the sequence of factorials equals an unsigned Stirling number of the first kind, 
$$
B_{m,r}(0!,1!,\cdots , (m-r)!)=|S(m,r)|={m \brack r} 
.$$
Likewise, by a numerical substitution into the definition of Bell polynomials, we have 
$$
B_{m,r}(0!,1!,\cdots , (m-r)!)=m!\sum_{\substack{k \\ \sum{i.y_{k,i}}=m \\ \sum{y_{k,i}}=r}}{\prod_{i=1}^{m-r+1}{\frac{1}{i^{y_{k,i}}(y_{k,i})!}}} 
.$$
Hence, by equating, we get 
$$
\sum_{\substack{k \\ \sum{i.y_{k,i}}=m \\ \sum{y_{k,i}}=r}}{\prod_{i=1}^{m-r+1}{\frac{1}{i^{y_{k,i}}(y_{k,i})!}}}
=\frac{1}{m!}{m \brack r}
.$$
From Lemma \ref{Lemma 4.1}, we know that the biggest integer that can appear in a partition of an integer $m$ using $r$ terms is $m-r+1$ (which means that $y_{k,m-r}= \cdots = y_{k,m}=0$). Thus, we get  
$$
\sum_{\substack{k \\ \sum{i.y_{k,i}}=m \\ \sum{y_{k,i}}=r}}{\prod_{i=1}^{m}{\frac{1}{i^{y_{k,i}}(y_{k,i})!}}}
=\sum_{\substack{k \\ \sum{i.y_{k,i}}=m \\ \sum{y_{k,i}}=r}}{\prod_{i=1}^{m-r+1}{\frac{1}{i^{y_{k,i}}(y_{k,i})!}}}
=\frac{1}{m!}{m \brack r}
.$$
\end{proof}
By Adding the arguments of the sum from Lemma \ref{Lemma 4.2} for all possible partition lengths, we obtain the following identity.  
\begin{lemma} \label{Lemma 4.3}
Let $m$ be a non-negative integer, the following sum over all partitions of $m$ can be shown to equal $1$ independently of the value of $m$, 
$$
\sum_{\substack{k \\ \sum{i.y_{k,i}}=m}}{\prod_{i=1}^{m}{\frac{1}{i^{y_{k,i}}(y_{k,i})!}}}
=1
.$$
\end{lemma}
\begin{proof}
From Lemma \ref{Lemma 4.2}, we have 
$$
\sum_{\substack{k \\ \sum{i.y_{k,i}}=m \\ \sum{y_{k,i}}=r}}{\prod_{i=1}^{m}{\frac{1}{i^{y_{k,i}}(y_{k,i})!}}}
=\frac{1}{m!}{m \brack r}
.$$
Hence, 
$$
\sum_{\substack{k \\ \sum{i.y_{k,i}}=m}}{\prod_{i=1}^{m}{\frac{1}{i^{y_{k,i}}(y_{k,i})!}}}
=\sum_{r=0}^{m}{\sum_{\substack{k \\ \sum{i.y_{k,i}}=m \\ \sum{y_{k,i}}=r}}{\prod_{i=1}^{m}{\frac{1}{i^{y_{k,i}}(y_{k,i})!}}}}
=\sum_{r=0}^{m}{\frac{1}{m!}{m \brack r}}
=\frac{1}{m!}\sum_{r=0}^{m}{{m \brack r}}
.$$
However, we have already shown that the finite sum of ${m \brack r}$ is equal to $m!$. Hence, 
$$
\sum_{\substack{k \\ \sum{i.y_{k,i}}=m}}{\prod_{i=1}^{m}{\frac{1}{i^{y_{k,i}}(y_{k,i})!}}}
=1
.$$
\end{proof}
A more general form of Lemma \ref{Lemma 4.2} is illustrated in the following lemma.  
\begin{lemma} \label{Lemma 4.4}
Let $(\varphi_1, \cdots , \varphi_m)$ be a partition of $\varphi \leq m$ such that $ \sum{\varphi_i}=r_{\varphi}$. Let $(y_{k,1}, \cdots , y_{k,m})=\{(y_{1,1}, \cdots , y_{1,m}),(y_{2,1}, \cdots , y_{2,m}), \cdots \}$ be the set of all partitions of $m$. 
$$
\sum_{\substack{k \\ \sum{i.y_{k,i}}=m \\ \sum{y_{k,i}}=r}}{\prod_{i=1}^{m}{\frac{\binom{y_{k,i}}{\varphi_i}}{i^{y_{k,i}}(y_{k,i})!}}}
=\sum_{\substack{k \\ \sum{i.y_{k,i}}=m \\ \sum{y_{k,i}}=r \\ y_{k,i} \geq \varphi_i}}{\prod_{i=1}^{m}{\frac{\binom{y_{k,i}}{\varphi_i}}{i^{y_{k,i}}(y_{k,i})!}}}
=\frac{1}{(m-\varphi)!}{m-\varphi \brack r-r_{\varphi}} \prod_{i=1}^{m}{\frac{1}{i^{\varphi_{i}} (\varphi_{i})!}}
.$$
\begin{remark}
{\em Knowing that the largest element of a partition of $\varphi$ is $\varphi$, we can rewrite it as follows }
$$
\sum_{\substack{k \\ \sum{i.y_{k,i}}=m \\ \sum{y_{k,i}}=r}}{\prod_{i=1}^{m}{\frac{\binom{y_{k,i}}{\varphi_i}}{i^{y_{k,i}}(y_{k,i})!}}}
=\sum_{\substack{k \\ \sum{i.y_{k,i}}=m \\ \sum{y_{k,i}}=r \\ y_{k,i} \geq \varphi_i}}{\prod_{i=1}^{m}{\frac{\binom{y_{k,i}}{\varphi_i}}{i^{y_{k,i}}(y_{k,i})!}}}
=\frac{1}{(m-\varphi)!}{m-\varphi \brack r-r_{\varphi}} \prod_{i=1}^{\varphi}{\frac{1}{i^{\varphi_{i}} (\varphi_{i})!}}
.$$
\end{remark}
\end{lemma}
\begin{proof}
Knowing that $\binom{n}{k}$ is zero if $n<k$, then $\binom{y_{k,i}}{\varphi_i}=0$ if $ \exists i \in \mathbb{N}, y_{k,i}<\varphi_i$. Hence, 
$$
\sum_{\substack{k \\ \sum{i.y_{k,i}}=m \\ \sum{y_{k,i}}=r}}{\prod_{i=1}^{m}{\frac{\binom{y_{k,i}}{\varphi_i}}{i^{y_{k,i}}(y_{k,i})!}}}
=\sum_{\substack{k \\ \sum{i.y_{k,i}}=m \\ \sum{y_{k,i}}=r \\ \exists i, y_{k,i}<\varphi_i}}{\prod_{i=1}^{m}{\frac{\binom{y_{k,i}}{\varphi_i}}{i^{y_{k,i}}(y_{k,i})!}}}
+\sum_{\substack{k \\ \sum{i.y_{k,i}}=m \\ \sum{y_{k,i}}=r \\ y_{k,i} \geq \varphi_i}}{\prod_{i=1}^{m}{\frac{\binom{y_{k,i}}{\varphi_i}}{i^{y_{k,i}}(y_{k,i})!}}}
=\sum_{\substack{k \\ \sum{i.y_{k,i}}=m \\ \sum{y_{k,i}}=r \\ y_{k,i} \geq \varphi_i}}{\prod_{i=1}^{m}{\frac{\binom{y_{k,i}}{\varphi_i}}{i^{y_{k,i}}(y_{k,i})!}}}
.$$
The first part of the proof is complete. 
\begin{equation*}
\begin{split}
\sum_{\substack{k \\ \sum{i.y_{k,i}}=m \\ \sum{y_{k,i}}=r}}{\prod_{i=1}^{m}{\frac{\binom{y_{k,i}}{\varphi_i}}{i^{y_{k,i}}(y_{k,i})!}}}
&=\sum_{\substack{k \\ \sum{i.y_{k,i}}=m \\ \sum{y_{k,i}}=r}}{\prod_{i=1}^{m}{\frac{1}{i^{y_{k,i}}(y_{k,i})!}.\frac{y_{k,i}!}{\varphi_i! (y_{k,i}-\varphi_i)!}}}\\
&=\sum_{\substack{k \\ \sum{i.y_{k,i}}=m \\ \sum{y_{k,i}}=r}}{\prod_{i=1}^{m}{\frac{1}{i^{y_{k,i}}}.\frac{1}{\varphi_i! (y_{k,i}-\varphi_i)!}}} \\
&=\sum_{\substack{k \\ \sum{i.y_{k,i}}=m \\ \sum{y_{k,i}}=r}}{\prod_{i=1}^{m}{\frac{1}{i^{\varphi_i} \varphi_i!}.\frac{1}{i^{y_{k,i}-\varphi_i} (y_{k,i}-\varphi_i)!}}} \\
&=\sum_{\substack{k \\ \sum{i.y_{k,i}}=m \\ \sum{y_{k,i}}=r}}{\prod_{i=1}^{m}{\frac{1}{i^{\varphi_i} \varphi_i!}}\prod_{i=1}^{m}{\frac{1}{i^{y_{k,i}-\varphi_i} (y_{k,i}-\varphi_i)!}}}
.\end{split}
\end{equation*}
As $\varphi_1, \cdots , \varphi_m$ are all constants then $\prod_{i=1}^{m}{\frac{1}{i^{\varphi_i} \varphi_i!}}$ is constant. This factor is constant and is common to all terms of the sum, therefore, we can factor it and take it outside the sum. 
$$
\sum_{\substack{k \\ \sum{i.y_{k,i}}=m \\ \sum{y_{k,i}}=r}}{\prod_{i=1}^{m}{\frac{\binom{y_{k,i}}{\varphi_i}}{i^{y_{k,i}}(y_{k,i})!}}}
=\left(\prod_{i=1}^{m}{\frac{1}{i^{\varphi_i} \varphi_i!}}\right) \sum_{\substack{k \\ \sum{i.y_{k,i}}=m \\ \sum{y_{k,i}}=r}}{\prod_{i=1}^{m}{\frac{1}{i^{y_{k,i}-\varphi_i} (y_{k,i}-\varphi_i)!}}}
.$$ 
Having that $(\varphi_1, \cdots , \varphi_m)$ is a partition of $\varphi \leq m$, hence, $\sum{i.\varphi_i}=\varphi \leq m$. Thus, the condition $\sum{i.y_{k,i}}=m$ can be replaced by $\sum{i.(y_{k,i}-\varphi_i)}=\sum{i.y_{k,i}}-\sum{i.\varphi_i}=m-\varphi$. Similarly, $r_{\varphi}=\sum{\varphi_i}$, hence, the condition $\sum{y_{k,i}}=r$ can be replaced by $\sum{(y_{k,i}-\varphi_i)}=\sum{y_{k,i}}-\sum{\varphi_i}=r-r_{\varphi}$. Hence, 
$$
\sum_{\substack{k \\ \sum{i.y_{k,i}}=m \\ \sum{y_{k,i}}=r}}{\prod_{i=1}^{m}{\frac{\binom{y_{k,i}}{\varphi_i}}{i^{y_{k,i}}(y_{k,i})!}}}
=\left(\prod_{i=1}^{m}{\frac{1}{i^{\varphi_i} \varphi_i!}}\right) \sum_{\substack{k \\ \sum{i.(y_{k,i}-\varphi_i)}=m-\varphi \\ \sum{(y_{k,i}-\varphi)}=r-r_{\varphi}}}{\prod_{i=1}^{m}{\frac{1}{i^{y_{k,i}-\varphi_i} (y_{k,i}-\varphi_i)!}}}
.$$ 
Let $Y_{k,i}=y_{k,i}-\varphi_i$, 
$$
\sum_{\substack{k \\ \sum{i.y_{k,i}}=m \\ \sum{y_{k,i}}=r}}{\prod_{i=1}^{m}{\frac{\binom{y_{k,i}}{\varphi_i}}{i^{y_{k,i}}(y_{k,i})!}}}
=\left(\prod_{i=1}^{m}{\frac{1}{i^{\varphi_i} \varphi_i!}}\right) \sum_{\substack{k \\ \sum{i.Y_{k,i}}=m-\varphi \\ \sum{Y_{k,i}}=r-r_{\varphi}}}{\prod_{i=1}^{m}{\frac{1}{i^{Y_{k,i}} Y_{k,i}!}}}
.$$
Knowing that the largest element of a partition of $(m-\varphi)$ is $(m-\varphi)$, hence,   
$$
\sum_{\substack{k \\ \sum{i.y_{k,i}}=m \\ \sum{y_{k,i}}=r}}{\prod_{i=1}^{m}{\frac{\binom{y_{k,i}}{\varphi_i}}{i^{y_{k,i}}(y_{k,i})!}}}
=\left(\prod_{i=1}^{m}{\frac{1}{i^{\varphi_i} \varphi_i!}}\right) \sum_{\substack{k \\ \sum{i.Y_{k,i}}=m-\varphi \\ \sum{Y_{k,i}}=r-r_{\varphi}}}{\prod_{i=1}^{m-\varphi}{\frac{1}{i^{Y_{k,i}} Y_{k,i}!}}}
.$$
Applying Lemma \ref{Lemma 4.2}, with $y_{k,i}$ substituted by $Y_{k,i}$, $m$ substituted by $m-\varphi$, and $r$ substituted by $r-r_{\varphi}$, we get 
$$
\sum_{\substack{k \\ \sum{i.y_{k,i}}=m \\ \sum{y_{k,i}}=r}}{\prod_{i=1}^{m}{\frac{\binom{y_{k,i}}{\varphi_i}}{i^{y_{k,i}}(y_{k,i})!}}}
=\left( \prod_{i=1}^{m}{\frac{1}{i^{\varphi_{i}} \varphi_{i}!}} \right)\frac{1}{(m-\varphi)!}{m-\varphi \brack r-r_{\varphi}}
.$$
The proof is complete. 
\end{proof}
\begin{remark}
{\em If $\varphi>m$, then $\sum{i.Y_{k,i}}=m-\varphi<0$ which makes Lemma \ref{Lemma 4.2} invalid which then makes this lemma invalid. }
\end{remark}
Similarly, a more general form of Lemma \ref{Lemma 4.3} is illustrated in the following lemma.  
\begin{lemma} \label{Lemma 4.5}
Let $(y_{k,1}, \cdots , y_{k,m})=\{(y_{1,1}, \cdots , y_{1,m}),(y_{2,1}, \cdots , y_{2,m}), \cdots \}$ be the set of all partitions of $m$. Let $(\varphi_1, \cdots , \varphi_m)$ be a partition of $r \leq m$.
$$
\sum_{\substack{k \\ \sum{i.y_{k,i}}=m}}{\prod_{i=1}^{m}{\frac{\binom{y_{k,i}}{\varphi_i}}{i^{y_{k,i}}(y_{k,i})!}}}
=\sum_{\substack{k \\ \sum{i.y_{k,i}}=m \\ y_{k,i} \geq \varphi_i}}{\prod_{i=1}^{m}{\frac{\binom{y_{k,i}}{\varphi_i}}{i^{y_{k,i}}(y_{k,i})!}}}
=\prod_{i=1}^{m}{\frac{1}{i^{\varphi_{i}} (\varphi_{i})!}}
.$$
\begin{remark}
{\em Knowing that the largest element of a partition of $r$ is $r$, we can rewrite it as follows }
$$
\sum_{\substack{k \\ \sum{i.y_{k,i}}=m}}{\prod_{i=1}^{m}{\frac{\binom{y_{k,i}}{\varphi_i}}{i^{y_{k,i}}(y_{k,i})!}}}
=\sum_{\substack{k \\ \sum{i.y_{k,i}}=m \\ y_{k,i} \geq \varphi_i}}{\prod_{i=1}^{m}{\frac{\binom{y_{k,i}}{\varphi_i}}{i^{y_{k,i}}(y_{k,i})!}}}
=\prod_{i=1}^{r}{\frac{1}{i^{\varphi_{i}} (\varphi_{i})!}}
.$$
\end{remark}
\end{lemma}
\begin{proof}
Knowing that $\binom{n}{k}$ is zero if $n<k$, then $\binom{y_{k,i}}{\varphi_i}=0$ if $ \exists i \in \mathbb{N}, y_{k,i}<\varphi_i$. Hence, 
$$
\sum_{\substack{k \\ \sum{i.y_{k,i}}=m}}{\prod_{i=1}^{m}{\frac{\binom{y_{k,i}}{\varphi_i}}{i^{y_{k,i}}(y_{k,i})!}}}
=\sum_{\substack{k \\ \sum{i.y_{k,i}}=m \\ \exists i, y_{k,i}<\varphi_i}}{\prod_{i=1}^{m}{\frac{\binom{y_{k,i}}{\varphi_i}}{i^{y_{k,i}}(y_{k,i})!}}}
+\sum_{\substack{k \\ \sum{i.y_{k,i}}=m \\ y_{k,i} \geq \varphi_i}}{\prod_{i=1}^{m}{\frac{\binom{y_{k,i}}{\varphi_i}}{i^{y_{k,i}}(y_{k,i})!}}}
=\sum_{\substack{k \\ \sum{i.y_{k,i}}=m \\ y_{k,i} \geq \varphi_i}}{\prod_{i=1}^{m}{\frac{\binom{y_{k,i}}{\varphi_i}}{i^{y_{k,i}}(y_{k,i})!}}}
.$$
The first part of the proof is complete. 
\begin{equation*}
\begin{split}
\sum_{\substack{k \\ \sum{i.y_{k,i}}=m}}{\prod_{i=1}^{m}{\frac{\binom{y_{k,i}}{\varphi_i}}{i^{y_{k,i}}(y_{k,i})!}}}
&=\sum_{\substack{k \\ \sum{i.y_{k,i}}=m}}{\prod_{i=1}^{m}{\frac{1}{i^{y_{k,i}}(y_{k,i})!}.\frac{y_{k,i}!}{\varphi_i! (y_{k,i}-\varphi_i)!}}}\\
&=\sum_{\substack{k \\ \sum{i.y_{k,i}}=m}}{\prod_{i=1}^{m}{\frac{1}{i^{y_{k,i}}}.\frac{1}{\varphi_i! (y_{k,i}-\varphi_i)!}}} \\
&=\sum_{\substack{k \\ \sum{i.y_{k,i}}=m}}{\prod_{i=1}^{m}{\frac{1}{i^{\varphi_i} \varphi_i!}.\frac{1}{i^{y_{k,i}-\varphi_i} (y_{k,i}-\varphi_i)!}}} \\
&=\sum_{\substack{k \\ \sum{i.y_{k,i}}=m}}{\prod_{i=1}^{m}{\frac{1}{i^{\varphi_i} \varphi_i!}}\prod_{i=1}^{m}{\frac{1}{i^{y_{k,i}-\varphi_i} (y_{k,i}-\varphi_i)!}}}
.\end{split}
\end{equation*}
As $\varphi_1, \cdots , \varphi_m$ are all constants then $\prod_{i=1}^{m}{\frac{1}{i^{\varphi_i} \varphi_i!}}$ is constant. This factor is constant and is common to all terms of the sum, therefore, we can factor it and take it outside the sum. 
$$
\sum_{\substack{k \\ \sum{i.y_{k,i}}=m}}{\prod_{i=1}^{m}{\frac{\binom{y_{k,i}}{\varphi_i}}{i^{y_{k,i}}(y_{k,i})!}}}
=\left(\prod_{i=1}^{m}{\frac{1}{i^{\varphi_i} \varphi_i!}}\right) \sum_{\substack{k \\ \sum{i.y_{k,i}}=m}}{\prod_{i=1}^{m}{\frac{1}{i^{y_{k,i}-\varphi_i} (y_{k,i}-\varphi_i)!}}}
.$$ 
Having that $(\varphi_1, \cdots , \varphi_m)$ is a partition of $r \leq m$, hence, $\sum{i.\varphi_i}=r \leq m$. Thus, the condition $\sum{i.y_{k,i}}=m$ can be replaced by $\sum{i.(y_{k,i}-\varphi_i)}=\sum{i.y_{k,i}}-\sum{i.\varphi_i}=m-r(\geq 0)$. Hence, 
$$
\sum_{\substack{k \\ \sum{i.y_{k,i}}=m}}{\prod_{i=1}^{m}{\frac{\binom{y_{k,i}}{\varphi_i}}{i^{y_{k,i}}(y_{k,i})!}}}
=\left(\prod_{i=1}^{m}{\frac{1}{i^{\varphi_i} \varphi_i!}}\right) \sum_{\substack{k \\ \sum{i.(y_{k,i}-\varphi_i)}=m-r}}{\prod_{i=1}^{m}{\frac{1}{i^{y_{k,i}-\varphi_i} (y_{k,i}-\varphi_i)!}}}
.$$ 
Let $Y_{k,i}=y_{k,i}-\varphi_i$, 
$$
\sum_{\substack{k \\ \sum{i.y_{k,i}}=m}}{\prod_{i=1}^{m}{\frac{\binom{y_{k,i}}{\varphi_i}}{i^{y_{k,i}}(y_{k,i})!}}}
=\left(\prod_{i=1}^{m}{\frac{1}{i^{\varphi_i} \varphi_i!}}\right) \sum_{\substack{k \\ \sum{i.Y_{k,i}}=m-r}}{\prod_{i=1}^{m}{\frac{1}{i^{Y_{k,i}} Y_{k,i}!}}}
.$$ 
Knowing that the largest element of a partition of $(m-r)$ is $(m-r)$, hence,   
$$
\sum_{\substack{k \\ \sum{i.y_{k,i}}=m}}{\prod_{i=1}^{m}{\frac{\binom{y_{k,i}}{\varphi_i}}{i^{y_{k,i}}(y_{k,i})!}}}
=\left(\prod_{i=1}^{m}{\frac{1}{i^{\varphi_i} \varphi_i!}}\right) \sum_{\substack{k \\ \sum{i.Y_{k,i}}=m-r}}{\prod_{i=1}^{m-r}{\frac{1}{i^{Y_{k,i}} Y_{k,i}!}}}
.$$ 
Applying Lemma \ref{Lemma 4.3}, with $y_{k,i}$ substituted by $Y_{k,i}$ and $m$ substituted by $m-r$, we get 
$$
\sum_{\substack{k \\ \sum{i.y_{k,i}}=m}}{\prod_{i=1}^{m}{\frac{\binom{y_{k,i}}{\varphi_i}}{i^{y_{k,i}}(y_{k,i})!}}}
=\left( \prod_{i=1}^{m}{\frac{1}{i^{\varphi_{i}} \varphi_{i}!}} \right)
.$$
The proof is complete. 
\end{proof}
\begin{remark}
{\em If $r>m$, then $\sum{i.Y_{k,i}}=m-r<0$ which makes Lemma \ref{Lemma 4.3} invalid which then makes this lemma invalid.}
\end{remark}
\begin{proposition} \label{Proposition 4.1}
Let $B_{m,r}(x_1, \cdots , x_{m-r+1})$ be the partial Bell polynomial and $B_m(x_1, \cdots , x_m)$ be the complete Bell polynomial, 
$$
\sum_{\sum{i.y_{k,i}}=m}{\prod_{i=1}^{m}{\frac{1}{(y_{k,i})!} \left( \frac{1}{i}\sum_{N=q}^{n}{(a_N)^i }\right)^{y_{k,i}}}}
=\frac{1}{m!}\sum_{r=0}^{m}{B_{m,r}(x_1, \cdots , x_{m-r+1})}
=\frac{1}{m!}B_m(x_1, \cdots , x_m)
$$
where $x_i=(i-1)!(\sum_{N=q}^{n}{(a_N)^i})$. 
\end{proposition}
\begin{proof}
From Lemma \ref{Lemma 4.1}, we can write 
$$
\sum_{\substack{\sum{i.y_{k,i}}=m \\ \sum{y_{k,i}}=r}}{\prod_{i=1}^{m}{\frac{1}{(y_{k,i})!} \left( \frac{1}{i}\sum_{N=q}^{n}{(a_N)^i }\right)^{y_{k,i}}}}
=\sum_{\substack{\sum{i.y_{k,i}}=m \\ \sum{y_{k,i}}=r}}{\prod_{i=1}^{m-r+1}{\frac{1}{(y_{k,i})!} \left( \frac{1}{i}\sum_{N=q}^{n}{(a_N)^i }\right)^{y_{k,i}}}}
.$$
We can notice that the right side term of the previous expression corresponds to a multiple of a special value of the partial Bell polynomial where $x_i=(i-1)!(\sum_{N=q}^{n}{(a_N)^i}), \forall i \in [1,m]$. Hence, 
$$
\sum_{\substack{\sum{i.y_{k,i}}=m \\ \sum{y_{k,i}}=r}}{\prod_{i=1}^{m}{\frac{1}{(y_{k,i})!} \left( \frac{1}{i}\sum_{N=q}^{n}{(a_N)^i }\right)^{y_{k,i}}}}
=\frac{1}{m!}B_{m,r}(x_1, \cdots , x_{m-r+1})
.$$
Additionally, the sum over the partitions of $m$ is equivalent to the sum for $r$ going from $0$ to $m$ of the sums over the  partitions of $m$ of length $r$. Thus, 
\begin{dmath*}
\sum_{\sum{i.y_{k,i}}=m}{\prod_{i=1}^{m}{\frac{1}{(y_{k,i})!} \left( \frac{1}{i}\sum_{N=q}^{n}{(a_N)^i }\right)^{y_{k,i}}}}
=\sum_{r=0}^{m}{\sum_{\substack{\sum{i.y_{k,i}}=m \\ \sum{y_{k,i}}=r}}{\prod_{i=1}^{m}{\frac{1}{(y_{k,i})!} \left( \frac{1}{i}\sum_{N=q}^{n}{(a_N)^i }\right)^{y_{k,i}}}}}
=\frac{1}{m!}\sum_{r=0}^{m}{B_{m,r}(x_1, \cdots , x_{m-r+1})}
.\end{dmath*}
Applying the definition of a complete Bell polynomial, we get 
 \begin{dmath*}
\sum_{\sum{i.y_{k,i}}=m}{\prod_{i=1}^{m}{\frac{1}{(y_{k,i})!} \left( \frac{1}{i}\sum_{N=q}^{n}{(a_N)^i }\right)^{y_{k,i}}}}
=\frac{1}{m!}B_m(x_1, \cdots , x_m)
.\end{dmath*}
\end{proof}
Now that all the required lemmas have been proven, we show the following theorem which allows the representation of a recurrent sum in terms of non-recurrent sums.  
\begin{theorem}[Reduction Theorem] \label{Theorem 4.1}
Let $m$ be a non-negative integer, $k$ be the index of the $k$-th partition of $m$ $(1 \leq k \leq p(m))$, $i$ be an integer between $1$ and $m$, and $y_{k,i}$ be the multiplicity of $i$ in the $k$-th partition of $m$. The reduction theorem for recurrent sums is stated as follow:
$$\sum_{N_m=q}^{n}{\cdots \sum_{N_1=q}^{N_2}{a_{N_m}\cdots a_{N_1}}}
=\sum_{\substack{k \\ \sum{i.y_{k,i}}=m}}{\prod_{i=1}^{m}{\frac{1}{(y_{k,i})!} \left( \frac{1}{i}\sum_{N=q}^{n}{(a_N)^i }\right)^{y_{k,i}}}}
.$$
\end{theorem}
\begin{proof}
1. Base Case: verify true for $n=q$, $\forall m \in \mathbb{N}$.
\begin{equation*}
\begin{split}
\sum_{\sum{i.y_{k,i}}=m}{\prod_{i=1}^{m}{\frac{1}{(y_{k,i})!} \left( \frac{1}{i}\sum_{N=q}^{q}{(a_N)^i }\right)^{y_{k,i}}}}
&=\sum_{\sum{i.y_{k,i}}=m}{\prod_{i=1}^{m}{\frac{1}{(y_{k,i})!i^{y_{k,i}}} \left(a_q\right)^{i.y_{k,i}}}} \\
&=\sum_{\sum{i.y_{k,i}}=m}{\left(a_q\right)^{\sum{i.y_{k,i}}}\prod_{i=1}^{m}{\frac{1}{(y_{k,i})!i^{y_{k,i}}}}} \\
&=\left(a_q\right)^{m}\sum_{\sum{i.y_{k,i}}=m}{\prod_{i=1}^{m}{\frac{1}{(y_{k,i})!i^{y_{k,i}}} }}
.\end{split}
\end{equation*}
By applying Lemma \ref{Lemma 4.3}, we get 
$$
\sum_{\sum{i.y_{k,i}}=m}{\prod_{i=1}^{m}{\frac{1}{(y_{k,i})!} \left( \frac{1}{i}\sum_{N=q}^{q}{(a_N)^i }\right)^{y_{k,i}}}}
=\left(a_q\right)^{m}
.$$
Likewise, 
$$
\sum_{N_m=q}^{q}{\cdots \sum_{N_1=q}^{N_2}{a_{N_m}\cdots a_{N_1}}}
=a_q \cdots a_q
=\left(a_q\right)^{m}
.$$
2. Induction hypothesis: assume the statement is true until $n$, $\forall m \in \mathbb{N}$.
$$\sum_{N_m=q}^{n}{\cdots \sum_{N_1=q}^{N_2}{a_{N_m}\cdots a_{N_1}}}
=\sum_{\sum{i.y_{k,i}}=m}{\prod_{i=1}^{m}{\frac{1}{(y_{k,i})!} \left( \frac{1}{i}\sum_{N=q}^{n}{(a_N)^i }\right)^{y_{k,i}}}}
.$$
3. Induction step: we will show that this statement is true for $(n+1)$, $\forall m \in \mathbb{N}$. \\
We have to show the following statement to be true: 
$$\sum_{N_m=q}^{n+1}{\cdots \sum_{N_1=q}^{N_2}{a_{N_m}\cdots a_{N_1}}}
=\sum_{\sum{i.y_{k,i}}=m}{\prod_{i=1}^{m}{\frac{1}{(y_{k,i})!} \left( \frac{1}{i}\sum_{N=q}^{n+1}{(a_N)^i }\right)^{y_{k,i}}}}
.$$
$ \\ $
$$
\sum_{\sum{i.y_{k,i}}=m}{\prod_{i=1}^{m}{\frac{1}{(y_{k,i})!} \left( \frac{1}{i}\sum_{N=q}^{n+1}{(a_N)^i }\right)^{y_{k,i}}}}
=\sum_{\sum{i.y_{k,i}}=m}{\prod_{i=1}^{m}{\frac{1}{(y_{k,i})!i^{y_{k,i}}} \left(\sum_{N=q}^{n}{(a_N)^i}+(a_{n+1})^i\right)^{y_{k,i}}}}
.$$
The binomial theorem states that 
$$
(a+b)^n=\sum_{\varphi=0}^{n}{\binom{n}{\varphi}a^{n-\varphi}b^{\varphi}}
.$$
Hence, 
\begin{equation*}
\begin{split}
\left(\sum_{N=q}^{n}{(a_N)^i}+(a_{n+1})^i\right)^{y_{k,i}}
&=\sum_{\varphi=0}^{y_{k,i}}{\binom{y_{k,i}}{\varphi}{\left(\sum_{N=q}^{n}{(a_N)^i}\right)}^{\varphi}{\left((a_{n+1})^i\right)}^{y_{k,i}-\varphi}}.
\end{split}
\end{equation*}
Thus, 
\begin{dmath*}
\sum_{\sum{i.y_{k,i}}=m}{\prod_{i=1}^{m}{\frac{1}{(y_{k,i})!} \left( \frac{1}{i}\sum_{N=q}^{n+1}{(a_N)^i }\right)^{y_{k,i}}}}
=\sum_{\sum{i.y_{k,i}}=m}{\prod_{i=1}^{m}{\frac{1}{(y_{k,i})!i^{y_{k,i}}}\sum_{\varphi=0}^{y_{k,i}}{\binom{y_{k,i}}{\varphi}{\left(\sum_{N=q}^{n}{(a_N)^i}\right)}^{\varphi}{\left((a_{n+1})^i\right)}^{y_{k,i}-\varphi}}}}
=\sum_{\sum{i.y_{k,i}}=m}{\prod_{i=1}^{m}{\sum_{\varphi=0}^{y_{k,i}}{\frac{1}{(y_{k,i})!i^{y_{k,i}}}\binom{y_{k,i}}{\varphi}{\left(\sum_{N=q}^{n}{(a_N)^i}\right)}^{\varphi}{\left(a_{n+1}\right)}^{i.y_{k,i}-i.\varphi}}}}
.\end{dmath*}
Let $A_{\varphi,i,k}={\frac{1}{(y_{k,i})!i^{y_{k,i}}}\binom{y_{k,i}}{\varphi}{\left(\sum_{N=q}^{n}{(a_N)^i}\right)}^{\varphi}{\left(a_{n+1}\right)}^{i.y_{k,i}-i.\varphi}}$.
By expanding then regrouping, it can be seen that 
$$
\prod_{i=1}^{m}{\sum_{\varphi=0}^{y_{k,i}}{A_{\varphi,i,k}}}
=\sum_{\varphi_{m}=0}^{y_{k,m}} {\cdots \sum_{\varphi_{1}=0}^{y_{k,1}}{\prod_{i=1}^{m}{A_{\varphi_{i},i,k}}}}
.$$
This is because, for any given $k$, by expanding the product of sums (the left hand side term), we will get a sum of products of the form $A_{\varphi_1,1}A_{\varphi_2,2} \cdots A_{\varphi_m,m}$ ($\prod_{i=1}^{m}{A_{\varphi_i,i}}$) for all combinations of $\varphi_1,\varphi_2, \cdots ,\varphi_m$ such that $0 \leq \varphi_1 \leq y_{k,1}, \cdots ,0 \leq \varphi_m \leq y_{k,m}$, which is equivalent to the right hand side term. \\
Hence, 
\begin{dmath*}
\sum_{\sum{i.y_{k,i}}=m}{\prod_{i=1}^{m}{\frac{1}{(y_{k,i})!} \left( \frac{1}{i}\sum_{N=q}^{n+1}{(a_N)^i }\right)^{y_{k,i}}}}
=\sum_{\sum{i.y_{k,i}}=m}{\sum_{\varphi_{m}=0}^{y_{k,m}}{\cdots \sum_{\varphi_{1}=0}^{y_{k,1}}{\prod_{i=1}^{m}{\frac{1}{(y_{k,i})!i^{y_{k,i}}}\binom{y_{k,i}}{\varphi_{i}}{\left(\sum_{N=q}^{n}{(a_N)^i}\right)}^{\varphi_{i}}{\left(a_{n+1}\right)}^{i.y_{k,i}-i.\varphi_{i}}}}}}
.\end{dmath*}
A more compact way of writing the repeated sum over the $\varphi_i$'s is by expressing it with one sum that combines all the conditions. The set of conditions $0\leq \varphi_1 \leq y_{k,1}, \cdots ,0\leq \varphi_m \leq y_{k,m}$ can be expressed as the condition $0 \leq \varphi_i \leq y_{k,i}$ for $i \in [1,m]$. 
\begin{dmath*}
\sum_{\sum{i.y_{k,i}}=m}{\prod_{i=1}^{m}{\frac{1}{(y_{k,i})!} \left( \frac{1}{i}\sum_{N=q}^{n+1}{(a_N)^i }\right)^{y_{k,i}}}}
=\sum_{\sum{i.y_{k,i}}=m}{\sum_{0 \leq \varphi_{i} \leq y_{k,i}}{\prod_{i=1}^{m}{\frac{1}{(y_{k,i})!i^{y_{k,i}}}\binom{y_{k,i}}{\varphi_{i}}{\left(\sum_{N=q}^{n}{(a_N)^i}\right)}^{\varphi_{i}}{\left(a_{n+1}\right)}^{i.y_{k,i}-i.\varphi_{i}}}}}
.\end{dmath*}
Similarly, let $j$ represent $\sum{i.\varphi_{i}}$. Hence, we can add the trivial condition that is $j=\sum{i.\varphi_{i}}$ to the sum over $\varphi_i$. Additionally, \\
$\sum{i.\varphi_{i}}=j$ is minimal when $\varphi_1=0, \cdots , \varphi_m=0$. Hence $j_{min}=0$. \\  
$\sum{i.\varphi_{i}}=j$ is maximal when $\varphi_1=y_{k,1}, \cdots , \varphi_m=y_{k,m}$. Hence $j_{max}=\sum{i.y_{k,i}}=m$. \\ 
Therefore, we have that $0 \leq j \leq m$ or equivalently that $j$ can go from $0$ to $m$. Hence, knowing that adding a true statement to a condition does not change the condition, we can add this additional condition to get 
\begin{dmath*}
\sum_{\sum{i.y_{k,i}}=m}{\prod_{i=1}^{m}{\frac{1}{(y_{k,i})!} \left( \frac{1}{i}\sum_{N=q}^{n+1}{(a_N)^i }\right)^{y_{k,i}}}}
=\sum_{\sum{i.y_{k,i}}=m}{\sum_{\substack{j=0 \\ \sum {i.\varphi_i}=j \\ 0 \leq \varphi_{i} \leq y_{k,i}}}^{m}{ \prod_{i=1}^{m}{\frac{1}{(y_{k,i})!i^{y_{k,i}}}\binom{y_{k,i}}{\varphi_{i}}{\left(\sum_{N=q}^{n}{(a_N)^i}\right)}^{\varphi_{i}}{\left(a_{n+1}\right)}^{i.y_{k,i}-i.\varphi_{i}}}}}
.\end{dmath*}
Knowing that $\binom{y_{k,i}}{\varphi_i}=0$ if $\varphi_i>y_{k,i}$, hence, the terms produced for $\varphi_i>y_{k,i}$ would be zero. Thus, we can remove the condition $0 \leq \varphi_i \leq y_{k,i}$ because terms that do not satisfy this condition will be zeros and, therefore, would not change the value of the sum. 
\begin{dmath*}
\sum_{\sum{i.y_{k,i}}=m}{\prod_{i=1}^{m}{\frac{1}{(y_{k,i})!} \left( \frac{1}{i}\sum_{N=q}^{n+1}{(a_N)^i }\right)^{y_{k,i}}}}
=\sum_{\sum{i.y_{k,i}}=m}{\sum_{\substack{j=0 \\ \sum {i.\varphi_i}=j}}^{m}{ \prod_{i=1}^{m}{\frac{1}{(y_{k,i})!i^{y_{k,i}}}\binom{y_{k,i}}{\varphi_{i}}{\left(\sum_{N=q}^{n}{(a_N)^i}\right)}^{\varphi_{i}}{\left(a_{n+1}\right)}^{i.y_{k,i}-i.\varphi_{i}}}}}
.\end{dmath*}
We expand the expression then, from all values of $k$ (from every partitions $(y_{k,1}, \cdots , y_{k,m})$ of $m$), we regroup together the terms having a combination of exponents $(\varphi_1, \cdots , \varphi_m)$ that forms a partition of the same integer $j$ and we do so $\forall j \in [0,m]$. 
Hence, performing this manipulation allows us to interchange the sum over $k$ (over $\sum_i{i.y_{k,i}}=m$) with the sums over $j$.
Thus, the expression becomes as follows, 
\begin{dmath*}
\sum_{\sum{i.y_{k,i}}=m}{\prod_{i=1}^{m}{\frac{1}{(y_{k,i})!} \left( \frac{1}{i}\sum_{N=q}^{n+1}{(a_N)^i }\right)^{y_{k,i}}}}
=\sum_{\substack{j=0 \\ \sum {i.\varphi_i}=j}}^{m}{\sum_{\substack{\sum{i.y_{k,i}}=m}}{ \prod_{i=1}^{m}{\frac{1}{(y_{k,i})!i^{y_{k,i}}}\binom{y_{k,i}}{\varphi_{i}}{\left(\sum_{N=q}^{n}{(a_N)^i}\right)}^{\varphi_{i}}{\left(a_{n+1}\right)}^{i.y_{k,i}-i.\varphi_{i}}}}}
=\sum_{\substack{j=0 \\ \sum {i.\varphi_i}=j}}^{m}{\sum_{\substack{\sum{i.y_{k,i}}=m}}{{\left(a_{n+1}\right)}^{\sum{i.y_{k,i}}-\sum{i.\varphi_{i}}}\left[\prod_{i=1}^{m}{{\left(\sum_{N=q}^{n}{(a_N)^i}\right)}^{\varphi_{i}}}\right]\left[\prod_{i=1}^{m}{\frac{1}{(y_{k,i})!i^{y_{k,i}}}\binom{y_{k,i}}{\varphi_{i}}}\right]}}
=\sum_{\substack{j=0 \\ \sum {i.\varphi_i}=j}}^{m}{{\left(a_{n+1}\right)}^{m-j}\left[\prod_{i=1}^{m}{{\left(\sum_{N=q}^{n}{(a_N)^i}\right)}^{\varphi_{i}}}\right]\left(\sum_{\substack{\sum{i.y_{k,i}}=m}}{\prod_{i=1}^{m}{\frac{1}{(y_{k,i})!i^{y_{k,i}}}\binom{y_{k,i}}{\varphi_{i}}}}\right)}
.\end{dmath*}
Applying Lemma \ref{Lemma 4.5}, we get
\begin{dmath*}
\sum_{\sum{i.y_{k,i}}=m}{\prod_{i=1}^{m}{\frac{1}{(y_{k,i})!} \left( \frac{1}{i}\sum_{N=q}^{n+1}{(a_N)^i }\right)^{y_{k,i}}}}
=\sum_{\substack{j=0 \\ \sum {i.\varphi_i}=j}}^{m}{{\left(a_{n+1}\right)}^{m-j}\left[\prod_{i=1}^{m}{{\left(\sum_{N=q}^{n}{(a_N)^i}\right)}^{\varphi_{i}}}\right]\left(\prod_{i=1}^{m}{\frac{1}{i^{\varphi_{i}} (\varphi_{i})!}}\right)}
=\sum_{\substack{j=0 \\ \sum {i.\varphi_i}=j}}^{m}{{\left(a_{n+1}\right)}^{m-j}\left(\prod_{i=1}^{m}{\frac{1}{i^{\varphi_{i}} (\varphi_{i})!}}{\left(\sum_{N=q}^{n}{(a_N)^i}\right)}^{\varphi_{i}}\right)}
.\end{dmath*} 
Knowing that for any given value of $j$ there is  multiple combinations of $\varphi_1, \cdots , \varphi_m$ that satisfy $\sum{i.\varphi_i}=j$. Hence, every value of $j$ corresponds to a sum of the sum's argument for all partitions of $j$ (for all combinations of $\varphi_1, \cdots , \varphi_m$ satisfying $\sum{i.\varphi_i}=j$). Therefore, we can split the outer sum with two conditions into two sums each with one of the conditions as follows, 
\begin{dmath*}
\sum_{\sum{i.y_{k,i}}=m}{\prod_{i=1}^{m}{\frac{1}{(y_{k,i})!} \left( \frac{1}{i}\sum_{N=q}^{n+1}{(a_N)^i }\right)^{y_{k,i}}}}
=\sum_{j=0}^{m}{{\left(a_{n+1}\right)}^{m-j}\sum_{\sum {i.\varphi_i}=j}{\left(\prod_{i=1}^{m}{\frac{1}{i^{\varphi_{i}} (\varphi_{i})!}}{\left(\sum_{N=q}^{n}{(a_N)^i}\right)}^{\varphi_{i}}\right)}}
.\end{dmath*} 
Knowing that the largest element of a partition of $j$ is $j$, 
\begin{dmath*}
\sum_{\sum{i.y_{k,i}}=m}{\prod_{i=1}^{m}{\frac{1}{(y_{k,i})!} \left( \frac{1}{i}\sum_{N=q}^{n+1}{(a_N)^i }\right)^{y_{k,i}}}}
=\sum_{j=0}^{m}{{\left(a_{n+1}\right)}^{m-j}\left(\sum_{\sum {i.\varphi_i}=j}{\prod_{i=1}^{j}{\frac{1}{i^{\varphi_{i}} (\varphi_{i})!}}{\left(\sum_{N=q}^{n}{(a_N)^i}\right)}^{\varphi_{i}}}\right)}
.\end{dmath*} 
By using the induction hypothesis, the expression becomes 
$$
\sum_{\sum{i.y_{k,i}}=m}{\prod_{i=1}^{m}{\frac{1}{(y_{k,i})!} \left( \frac{1}{i}\sum_{N=q}^{n+1}{(a_N)^i }\right)^{y_{k,i}}}}
=\sum_{j=0}^{m}{\left(a_{n+1}\right)^{m-j}\left(\sum_{N_j=q}^{n}{\cdots \sum_{N_1=q}^{N_2}{a_{N_j} \cdots a_{N_1}}}\right)}
.$$
Using  Corollary \ref{Corollary 2.1}, we get 
$$
\sum_{\sum{i.y_{k,i}}=m}{\prod_{i=1}^{m}{\frac{1}{(y_{k,i})!} \left( \frac{1}{i}\sum_{N=q}^{n+1}{(a_N)^i }\right)^{y_{k,i}}}}
=\sum_{N_m=q}^{n+1}{\cdots \sum_{N_1=q}^{N_2}{a_{N_m}\cdots a_{N_1}}}
.$$
The theorem is proven by induction. 
\end{proof}
\begin{corollary} \label{Corollary 4.1}
If the recurrent sum starts at 1, Theorem \ref{Theorem 4.1} becomes 
$$\sum_{N_m=1}^{n}{\cdots \sum_{N_1=1}^{N_2}{a_{N_m}\cdots a_{N_1}}}
=\sum_{\sum{i.y_{k,i}}=m}{\prod_{i=1}^{m}{\frac{1}{(y_{k,i})!} \left( \frac{1}{i}\sum_{N=1}^{n}{(a_N)^i }\right)^{y_{k,i}}}}
.$$
\end{corollary}
An additional partition identity that can be deduced from Theorem \ref{Theorem 4.1} is as follows.  
\begin{corollary} \label{Corollary 4.2}
For any $m,n \in \mathbb{N}$, we have that 
$$
\sum_{\sum{i.y_{k,i}}=m}{\prod_{i=1}^{m}{\frac{1}{(y_{k,i})!} \left( \frac{n}{i}\right)^{y_{k,i}}}}=\binom{n+m-1}{m}
.$$
\end{corollary}
\begin{proof}
From paper \cite{RepeatedSums}, we have the following relation, 
$$
\sum_{N_m=1}^{n}{\cdots \sum_{N_1=1}^{N_2}{1}}=\binom{n+m-1}{m}
.$$
By applying Theorem \ref{Theorem 4.1}, we get 
$$
\sum_{\sum{i.y_{k,i}}=m}{\prod_{i=1}^{m}{\frac{1}{(y_{k,i})!} \left( \frac{1}{i}\sum_{N=1}^{n}{1}\right)^{y_{k,i}}}}
=\sum_{\sum{i.y_{k,i}}=m}{\prod_{i=1}^{m}{\frac{1}{(y_{k,i})!} \left( \frac{n}{i}\right)^{y_{k,i}}}}
=\binom{n+m-1}{m}
.$$
\end{proof}
\begin{example}
{\em For $n=1$, Corollary \ref{Corollary 4.2} gives }
$$
\sum_{\sum{i.y_{k,i}}=m}{\prod_{i=1}^{m}{\frac{1}{(y_{k,i})!i^{y_{k,i}}}}}
=\binom{m}{m}=1
.$$
\end{example}
\begin{example}
{\em For $n=2$, Corollary \ref{Corollary 4.2} gives }
$$
\sum_{\sum{i.y_{k,i}}=m}{\prod_{i=1}^{m}{\frac{2^{y_{k,i}}}{(y_{k,i})!i^{y_{k,i}}}}}
=\binom{m+1}{m}=m+1
.$$
\end{example}
\begin{example}
{\em For $n=3$, Corollary \ref{Corollary 4.2} gives }
$$
\sum_{\sum{i.y_{k,i}}=m}{\prod_{i=1}^{m}{\frac{3^{y_{k,i}}}{(y_{k,i})!i^{y_{k,i}}}}}
=\binom{m+2}{m}=\frac{(m+1)(m+2)}{2}
.$$
\end{example}
\subsection{Particular cases}
In this section, we will apply the reduction formula for the cases of $m$ from $1$ to $4$. These cases were independently proven using two distinct methods (which are omitted here for simplicity). Similarly, these formulas were verified for a certain range of $n$ using a computer program which calculated the right expression as well as the left expression then checks that they are equal. 
\begin{itemize}
\item For $m=1$ 
$$
\sum_{N_1=1}^{n}{a_{N_1}}
=\sum_{N=1}^{n}{a_{N}}
.$$
\item For $m=2$ 
$$
\sum_{N_2=1}^{n}{\sum_{N_1=1}^{N_2}{a_{N_2}a_{N_1}}}
=\frac{1}{2}\left(\sum_{N=1}^{n}{a_{N}}\right)^{2}+\frac{1}{2}\left(\sum_{N=1}^{n}{\left(a_{N}\right)^{2}}\right)
.$$
\item For $m=3$ 
$$
\sum_{N_3=1}^{n}{\sum_{N_2=1}^{N_3}{\sum_{N_1=1}^{N_2}{a_{N_3}a_{N_2}a_{N_1}}}}
=\frac{1}{6}\left(\sum_{N=1}^{n}{a_{N}}\right)^{3}
+\frac{1}{2}\left(\sum_{N=1}^{n}{a_{N}}\right)\left(\sum_{N=1}^{n}{\left(a_{N}\right)^{2}}\right)
+\frac{1}{3}\left(\sum_{N=1}^{n}{\left(a_{N}\right)^{3}}\right)
.$$
\item For $m=4$ 
\begin{dmath*}
\sum_{N_4=1}^{n}{\sum_{N_3=1}^{N_4}{\sum_{N_2=1}^{N_3}{\sum_{N_1=1}^{N_2}{a_{N_4}a_{N_3}a_{N_2}a_{N_1}}}}}
=\frac{1}{24}\left(\sum_{N=1}^{n}{a_{N}}\right)^{4}
+\frac{1}{4}\left(\sum_{N=1}^{n}{a_{N}}\right)^{2}\left(\sum_{N=1}^{n}{\left(a_{N}\right)^{2}}\right)
+\frac{1}{3}\left(\sum_{N=1}^{n}{a_{N}}\right)\left(\sum_{N=1}^{n}{\left(a_{N}\right)^{3}}\right)
+\frac{1}{8}\left(\sum_{N=1}^{n}{\left(a_{N}\right)^{2}}\right)^{2}
+\frac{1}{4}\left(\sum_{N=1}^{n}{\left(a_{N}\right)^{4}}\right)
.\end{dmath*}
\end{itemize}
\subsection{General Reduction Theorem}
We define the notation $|A|$ as the number of elements in the set $A$. Note that if $A$ is a set of sets then $|A|$ represents the number of sets in $A$ as they are considered the elements of $A$. \\ 
Let $m$ be a non-negative integer and let $\{(y_{1,1}, \cdots, y_{1,m}),(y_{2,1}, \cdots, y_{2,m}), \cdots \}$ be the set of all partitions of $m$. 
Let us consider the set $M=\{1, \cdots, m\}$. 
The permutation group $S_m$ is the set of all permutations of the set $\{1, \cdots, m \}$. Let $\sigma \in S_m$ be a permutation of the set $\{1, \cdots, m \}$ and let $\sigma (i)$ represent the $i$-th element of this given permutation. The number of such permutations is given by 
\begin{equation}
|S_m|=m!
.\end{equation}
The cycle-type of a permutation $\sigma$ is the ordered set where the $i$-th element represents the number of cycles of size $i$ in the cycle decomposition of $\sigma$. The number of ways of arranging $i$ elements cyclically is $(i-1)!$. The number of possible combinations of $y_{k,i}$ cycles of size $i$ is $[(i-1)!]^{y_{k,i}}$. Hence, the number of permutations having cycle-type $(y_{k,1}, \cdots, y_{k,m})$ is given by 
\begin{equation}
\prod_{i=1}^{m}{[(i-1)!]^{y_{k,i}}}
.\end{equation}
A partition $P$ of a set $M$ is a set of non-empty disjoint subsets of $M$ such that every element of $M$ is present in exactly one of the subsets. Let $P=\{\underbrace{P_{1,1}, \cdots, P_{1,y_1}}_{y_1 \,\, sets}, \cdots, \underbrace{P_{m,1}, \cdots, P_{m,y_m}}_{y_m \,\, sets}\}$ represent a partition of a set of $m$ elements (for our purpose let it be the set $\{1, \cdots, m\}$). $P_{i,y}$ represents the $y$-th subset of order (size) $i$. $y_i$ represents the number of subsets of size $i$ contained in this partition of the set. It is interesting to note that $(y_1, \cdots, y_m)$ will always form a partition of the non-negative integer $m$. However, the number of partitions of $m$ is different from the number of partitions of a set of $m$ elements because there are more than one partition of the set of $m$ elements that can be associated with a given partition of $m$. In fact, we can easily determine that the number of partitions of a set of $m$ elements associated with the partition $(y_1, \cdots, y_m)$ is given by 
\begin{equation} \label{setpartition}
|\Omega_k|
=\frac{m!}{1!^{y_{k,1}}\cdots m!^{y_{k,m}}(y_{k,1})! \cdots (y_{k,m})!}
=\frac{m!}{\prod_{i=1}^{m}{i!^{y_{k,i}}y_{k,i}!}}
.\end{equation}  
where $\Omega_k$ is the set of all partitions of the set of $m$ elements associated the partition $(y_{k,i}, \cdots, y_{k,m})$. This is because the number of ways to divide $m$ objects into $l_1$ groups of $1$ element, $l_2$ groups of $2$ elements, $\cdots$, and $l_m$ groups of $m$ elements is given by 
\begin{equation} 
\frac{m!}{1!^{l_1} \cdots m!^{l_{m}}l_1! \cdots l_m!}
=\frac{m!}{\prod_{i=1}^{m}{i!^{l_{i}}l_i!}}
.\end{equation}  
We will denote by $\Omega$ the set of all partitions of the set of $m$ elements. 
Finally, a partition $P$ of a set $M$ is a refinement of a partition $\rho$ of the same set $M$ if every element in $P$ is a subset of an element in $\rho$. We denote this as $P \succeq \rho$. \\
Using the notation introduced, we can formulate a generalization of Theorem \ref{Theorem 4.1} where all sequences are distinct. 
\begin{theorem} \label{generalReduction}
Let $m,n,q \in \mathbb{N}$ such that $n \geq q$. Let $a_{(1);N}, \cdots , a_{(m);N}$ be $m$ sequences defined in the interval $[q,n]$.
we have that 
\begin{equation*}
\begin{split}
\sum_{\sigma \in S_m}{\left(\sum_{N_{m}=q}^{n}{\cdots \sum_{N_1=q}^{N_2}{a_{(\sigma(m));N_{m}} \cdots a_{(\sigma(1));N_{1}} }}\right)} 
=\sum_{\substack{P \in \Omega}}{\prod_{i=1}^{m}{[(i-1)!]^{y_{k,i}} \left[\prod_{g=1}^{y_{k,i}}{\left( \sum_{N=q}^{n}{\prod_{h \in P_{i,g}}{a_{(h);N}}}\right)}\right]}}
.\end{split}
\end{equation*}
\end{theorem}
\begin{remark}
{\em The theorem can also be written as }
\begin{equation*}
\begin{split}
&\sum_{\sigma \in S_m}{\left(\sum_{N_{m}=q}^{n}{\cdots \sum_{N_1=q}^{N_2}{a_{(\sigma(m));N_{m}} \cdots a_{(\sigma(1));N_{1}} }}\right)} \\
&\,\,=\sum_{\substack{k \\ \sum{i.y_{k,i}}=m}}{\sum_{\Omega_{k}}{\prod_{i=1}^{m}{[(i-1)!]^{y_{k,i}} \left[\prod_{g=1}^{y_{k,i}}{\left( \sum_{N=q}^{n}{\prod_{h \in P_{i,g}}{a_{(h);N}}}\right)}\right]}}} \\
&\,\,=|S_m|\sum_{\substack{k \\ \sum{i.y_{k,i}}=m}}{\frac{1}{|\Omega_k|}\sum_{\Omega_{k}}{\prod_{i=1}^{m}{\frac{1}{y_{k,i}! i^{y_{k,i}}} \left[\prod_{g=1}^{y_{k,i}}{\left( \sum_{N=q}^{n}{\prod_{h\in P_{i,g}}{a_{(h);N}}}\right)}\right]}}} .
\end{split}
\end{equation*}
{\em As every partition of a set of $m$ elements is associated with a given partition of $m$, hence, adding up all the partitions of the set for ever given partition of $m$ is equivalent to adding up all partitions of the set. The first form is obtained by regrouping together, from the set of all partitions of the set $\{1, \cdots, m\}$, those who are associated with a given partition of $m$. \\
The second expression is obtained by noting that $\frac{|S_m|}{|\Omega_k|}\prod_{i=1}^{m}{\frac{1}{y_{k,i}!i^{y_{k,i}}}}=\prod_{i=1}^{m}{{[(i-1)!]^{y_{k,i}}}}$. These forms are shown as they can be more easily used to show that this theorem reduces to Theorem \ref{Theorem 4.1} if all sequences are the same. } 
\end{remark}
\begin{proof}
Both sides of the equation produce all combinations of terms which are products of the $m$ sequences. Hence, the strategy of this proof is to show that every combination appear with the same multiplicity on both sides. \\ We can assume the sequences to all be distinct without lost of generality. We can write 
\begin{equation*}
\sum_{\sigma \in S_m}{\left(\sum_{N_{m}=q}^{n}{\cdots \sum_{N_1=q}^{N_2}{a_{(\sigma(m));N_{m}} \cdots a_{(\sigma(1));N_{1}} }}\right)}
=\sum_{\sigma \in S_m}{\left(\sum_{N_{m}=q}^{n}{\cdots \sum_{N_1=q}^{N_2}{a_{(m);N_{\sigma(m)}} \cdots a_{(1);N_{\sigma(1)}} }}\right)}
.\end{equation*}
Hence, we can consider the symmetric group $S_m$ as acting on $N=(N_1, \cdots, N_m)$. $N=(N_1, \cdots, N_m)$ has an isotropy group $S_m(N)$ and an associated partition $\rho$ of the set of $m$ elements. The partition $\rho$ is the set of all equivalence classes of the relation given by $a \sim b$ if and only if $N_a=N_b$ and $S_m(N)=\{\sigma \in S_m \,\, | \,\, \sigma(i) \sim i \,\, \forall i\}$. Thus, 
\begin{equation} \label{seq}
a_{(m);N_m} \cdots a_{(1);N_1}
\end{equation}
appears $|S_m(N)|$ times in the expansion of the left hand side of the theorem. \\
Likewise, in the right hand side, \eqref{seq} can only appears in the terms corresponding to partitions $P$ which are refinements of $\rho$. \eqref{seq} appears 
\begin{equation}
\sum_{P \succeq \rho}{\prod_{i=1}^{m}{[(i-1)!]^{y_{k,i}}}}
\end{equation}
times in the right hand side of the theorem. Also let us notice that $[(i-1)!]^{y_{k,i}}$ corresponds to $(|P_{i,1}|-1)! \cdots (|P_{i,y_{k,i}}|-1)!$ because $|P_{i,1}|=\cdots=|P_{i,y_{k,i}}|=i$. Hence, $\prod_{i=1}^{m}{[(i-1)!]^{y_{k,i}}}$ corresponds to $\prod_{P_{h,g} \subset P}{(|P_{h,g}|-1)!}$ which is equal to the number of permutations having cycle-type specified by $P$. \\ 
Knowing that any element of $S_m(N)$ has a unique cycle-type specified by a partition that refines $\rho$, hence, we conclude that 
\begin{equation}
\sum_{P \succeq \rho}{\prod_{i=1}^{m}{[(i-1)!]^{y_{k,i}}}}
=|S_m(N)|
.\end{equation}
As both sides of the theorem produce the same terms and with the same multiplicity, we can say that these sides are equal to each other. 
\end{proof}
\begin{example}
{\em For $m=2$, Theorem \ref{generalReduction} gives the following, }
\begin{equation*}
\sum_{N_{2}=q}^{n}{ \sum_{N_1=q}^{N_2}{ a_{N_{2}}  b_{N_{1}} }}
+\sum_{N_2=q}^{n}{ \sum_{N_1=q}^{N_2}{ b_{N_{2}}  a_{N_{1}} }}
=\left(\sum_{N=q}^{n}{a_N}\right)\left(\sum_{N=q}^{n}{b_N}\right)
+\left(\sum_{N=q}^{n}{a_N b_N}\right)
.\end{equation*}
\end{example}
\begin{example}
{\em For $m=3$, Theorem \ref{generalReduction} gives the following, }
\begin{dmath*}
\sum_{\sigma \in S_3}{\left(\sum_{N_{3}=q}^{n}{\sum_{N_2=q}^{N_3}{ \sum_{N_1=q}^{N_2}{a_{(\sigma(3));N_{3}} a_{(\sigma(2));N_{2}}  a_{(\sigma(1));N_{1}} }}}\right)}
=\left(\sum_{N=q}^{n}{a_{(1);N}}\right)\left(\sum_{N=q}^{n}{a_{(2);N}}\right)\left(\sum_{N=q}^{n}{a_{(3);N}}\right)
+\left(\sum_{N=q}^{n}{a_{(1);N}}\right)\left(\sum_{N=q}^{n}{a_{(2);N} a_{(3);N}}\right)
+\left(\sum_{N=q}^{n}{a_{(2);N}}\right)\left(\sum_{N=q}^{n}{a_{(1);N} a_{(3);N}}\right)
+\left(\sum_{N=q}^{n}{a_{(3);N}}\right)\left(\sum_{N=q}^{n}{a_{(1);N} a_{(2);N}}\right)
+2\left(\sum_{N=q}^{n}{a_{(1);N} a_{(2);N} a_{(3);N}}\right)
.\end{dmath*}
\end{example}
\subsection{Example applications}
In this section, we will apply the reduction formula presented in Theorem \ref{Theorem 4.1} to simplify certain special recurrent sums. The first special sum that we will simplify is a recurrent sum of $N^p$ which will produce a recurrent form of the Faulhaber formula. The second special sum is the recurrent harmonic series as well as the recurrent $p$-series for positive even values of $p$. 
\subsubsection{Recurrent Faulhaber Formula}
The Faulhaber formula is a formula developed by Faulhaber in a 1631 edition of Academia Algebrae \cite{Faulhaber} to calculate sums of powers $(N^p)$. The Faulhaber formula is as follows 
\begin{equation}
\sum_{N=1}^{n}{N^p}=\frac{1}{p+1}\sum_{j=0}^{p}{(-1)^j\binom{p+1}{j}B_{j}n^{p+1-j}}
\end{equation} 
where $B_j$ are the Bernoulli numbers of the first kind. 
\begin{remark}
{\em See \cite{nielsen1923traite} for details on the history of Bernoulli numbers. }
\end{remark}
In this section, we will use the reduction formula for recurrent sums to develop a formula for a recurrent form of the Faulhaber formula. 

\begin{theorem} \label{Theorem 4.2}
For any $m,n,p \in \mathbb{N}$, we have that 
\begin{equation*}
\begin{split}
\sum_{N_m=1}^{n}{\cdots \sum_{N_1=1}^{N_2}{{N_m}^p\cdots {N_1}^p}}
&=\sum_{\sum{i.y_{k,i}}=m}{\prod_{i=1}^{m}{\frac{1}{(y_{k,i})!i^{y_{k,i}}} \left( \sum_{N=1}^{n}{N^{ip}}\right)^{y_{k,i}}}} \\
&=\sum_{\sum{i.y_{k,i}}=m}{\prod_{i=1}^{m}{\frac{1}{(y_{k,i})!i^{y_{k,i}}} \left(\frac{n^{ip+1}}{ip+1}\sum_{j=0}^{ip}{(-1)^j\binom{ip+1}{j}\frac{B_{j}}{n^{j}}} \right)^{y_{k,i}}}} \\
\end{split}
\end{equation*}
where $B_j$ are the Bernoulli numbers of the first kind. 
\end{theorem}
\begin{proof}
This theorem is obtained by applying Theorem \ref{Theorem 4.1} and then applying Faulhaber's formula. 
\end{proof}
\begin{corollary} \label{Corollary 4.3}
For $p=1$, Theorem \ref{Theorem 4.2} becomes 
\begin{equation*}
\begin{split}
\sum_{N_m=1}^{n}{\cdots \sum_{N_1=1}^{N_2}{{N_m}\cdots {N_1}}}
&=\sum_{\sum{i.y_{k,i}}=m}{\prod_{i=1}^{m}{\frac{1}{(y_{k,i})!i^{y_{k,i}}} \left( \sum_{N=1}^{n}{N^{i}}\right)^{y_{k,i}}}} \\
&=\sum_{\sum{i.y_{k,i}}=m}{\prod_{i=1}^{m}{\frac{1}{(y_{k,i})!i^{y_{k,i}}} \left(\frac{n^{i+1}}{i+1}\sum_{j=0}^{i}{(-1)^j\binom{i+1}{j}\frac{B_{j}}{n^{j}}} \right)^{y_{k,i}}}} 
.\end{split}
\end{equation*}
Where $B_j$ are the Bernoulli numbers of the first kind.
\end{corollary}
Let us now consider a few particular cases: 
\begin{itemize}
\item For $m=2$ 
\begin{equation*}
\begin{split}
&\sum_{N_2=1}^{n}{\sum_{N_1=1}^{N_2}{{N_2}^p {N_1}^p}} \\
&\,\,=\frac{1}{2}\left(\sum_{N=1}^{n}{{N}^p}\right)^{2}+\frac{1}{2}\left(\sum_{N=1}^{n}{{N}^{2p}}\right) \\
&\,\,=\frac{1}{2}\left[\left(\frac{n^{p+1}}{p+1}\sum_{j=0}^{p}{(-1)^j\binom{p+1}{j}\frac{B_{j}}{n^{j}}} \right)^{2}
+\left(\frac{n^{2p+1}}{2p+1}\sum_{j=0}^{2p}{(-1)^j\binom{2p+1}{j}\frac{B_{j}}{n^{j}}} \right)\right].
\end{split}
\end{equation*}
\begin{example}
{\em For $p=1$, by applying this theorem and exploiting Faulhaber's formula, we can get the following formula }
$$
\sum_{N_2=1}^{n}{\sum_{N_1=1}^{N_2}{{N_2}{N_1}}}
=\frac{n(n+1)(n+2)(3n+1)}{24}
.$$
\end{example}
\begin{example}
{\em For $p=2$, by applying this theorem and exploiting Faulhaber's formula, we can get the following formula }
$$
\sum_{N_2=1}^{n}{\sum_{N_1=1}^{N_2}{{N_2}^2 {N_1}^2}}
=\frac{n(n+1)(n+2)(2n+1)(2n+3)(5n-1)}{360}
.$$
\end{example}
\item For $m=3$  
$$
\sum_{N_3=1}^{n}{\sum_{N_2=1}^{N_3}{\sum_{N_1=1}^{N_2}{{N_3}^p{N_2}^p{N_1}^p}}}
=\frac{1}{6}\left(\sum_{N=1}^{n}{a_{N}}\right)^{3}
+\frac{1}{2}\left(\sum_{N=1}^{n}{a_{N}}\right)\left(\sum_{N=1}^{n}{\left(a_{N}\right)^{2}}\right)
+\frac{1}{3}\left(\sum_{N=1}^{n}{\left(a_{N}\right)^{3}}\right)
.$$
\begin{example}
{\em For $p=1$, by applying this theorem and exploiting Faulhaber's formula, we can get the following formula }
$$
\sum_{N_3=1}^{n}{\sum_{N_2=1}^{N_3}{\sum_{N_1=1}^{N_2}{{N_3}{N_2}{N_1}}}}
=\frac{n^2(n+1)^2(n+2)(n+3)}{48}
=\left(\sum_{N=1}^{n}{N}\right)\left[\frac{n(n+1)(n+2)(n+3)}{4!}\right]
.$$
\end{example}
\end{itemize}
\subsubsection{Recurrent p-series and harmonic series}
In this section, using the formula developed by Euler and the reduction theorem (Theorem \ref{Theorem 4.1}), we will prove an expression which can be used to calculate a recurrent form of the zeta function for positive even values. Then we will conjecture a solution for a more general form of the Basel problem.\\  

We start by applying Theorem \ref{Theorem 4.1} and using the expression of the zeta function for positive even values to get an expression for the recurrent series of $\frac{1}{N^{2p}}$ (or recurrent harmonic series). 
\begin{theorem} \label{Theorem 4.3}
For any $m,p \in \mathbb{N}$, we have that 
\begin{equation*}
\begin{split}
\sum_{N_{m}=1}^{\infty}{\cdots \sum_{N_1=1}^{N_2}{\frac{1}{N_{m}^{2p} \cdots N_{1}^{2p}}}}
&=\sum_{\sum{i.y_{k,i}}=m}{\prod_{i=1}^{m}{\frac{1}{(y_{k,i})!i^{y_{k,i}}} \left(\zeta(2ip)\right)^{y_{k,i}}}} \\
&=(-1)^{pm}(2\pi)^{2pm}\sum_{\sum{i.y_{k,i}}=m}{\prod_{i=1}^{m}{\frac{(-1)^{y_{k,i}}}{(y_{k,i})!} \left(\frac{B_{2ip}}{(2i)(2ip)!}\right)^{y_{k,i}}}}
.\end{split}
\end{equation*}
\end{theorem}
\begin{proof}
By applying Theorem \ref{Theorem 4.1}, 
\begin{equation*}
\begin{split}
\sum_{N_{m}=1}^{\infty}{\cdots \sum_{N_1=1}^{N_2}{\frac{1}{N_{m}^{2p} \cdots N_{1}^{2p}}}}
&=\sum_{\sum{i.y_{k,i}}=m}{\prod_{i=1}^{m}{\frac{1}{(y_{k,i})!i^{y_{k,i}}} \left(\sum_{N=1}^{\infty}{\left(\frac{1}{N^{2p}}\right)^{i}}\right)^{y_{k,i}}}} \\
&=\sum_{\sum{i.y_{k,i}}=m}{\prod_{i=1}^{m}{\frac{1}{(y_{k,i})!i^{y_{k,i}}} \left(\zeta(2ip)\right)^{y_{k,i}}}} 
.\end{split}
\end{equation*}
Euler proved that, for $m \geq 1$ (see \cite{arfken} for a proof),  
$$
\zeta(2m)=\frac{(-1)^{m+1}(2\pi)^{2m}}{2(2m)!}B_{2m}
.$$
Hence, 
\begin{equation*}
\begin{split}
\sum_{N_{m}=1}^{\infty}{\cdots \sum_{N_1=1}^{N_2}{\frac{1}{N_{m}^{2p} \cdots N_{1}^{2p}}}}
&=\sum_{\sum{i.y_{k,i}}=m}{\prod_{i=1}^{m}{\frac{1}{(y_{k,i})!i^{y_{k,i}}} \left((-1)^{ip+1}\frac{B_{2ip}(2\pi)^{2ip}}{2(2ip)!}\right)^{y_{k,i}}}} \\
&=(-1)^{pm}(2\pi)^{2pm}\sum_{\sum{i.y_{k,i}}=m}{\prod_{i=1}^{m}{\frac{(-1)^{y_{k,i}}}{(y_{k,i})!} \left(\frac{B_{2ip}}{(2i)(2ip)!}\right)^{y_{k,i}}}}
.\end{split}
\end{equation*}
\end{proof}
The following table summarizes some values of the zeta function for positive even arguments, 
$$
\zeta(2)=\frac{\pi^2}{6} \,\,\,\,\,\,
\zeta(4)=\frac{\pi^4}{90} \,\,\,\,\,\,
\zeta(6)=\frac{\pi^6}{945} \,\,\,\,\,\,
\zeta(8)=\frac{\pi^8}{9450} \,\,\,\,\,\,
\zeta(10)=\frac{\pi^{10}}{93555} \,\,\,\,\,\,
$$
$$
\zeta(12)=\frac{691\pi^{12}}{638512875} \,\,\,\,\,\,
\zeta(14)=\frac{2\pi^{14}}{18243225} \,\,\,\,\,\,
\zeta(16)=\frac{3617\pi^{16}}{325641566250} \,\,\,\,\,\,
.$$
By using the values in the above table as well as Theorem \ref{Theorem 4.3} and playing with different values, we can notice some identities. In particular, we can conjecture the following statement for the recurrent sum of $\frac{1}{N^2}$ (recurrent harmonic series with $2p=2$) for different values of $m$ (for different numbers of summations). This represents a generalization of the Basel Problem solved by Euler. However, this conjecture has already been proven, hence, we will directly use it to develop additional identities.  
\begin{theorem} \label{Conjecture 4.1}
For any $m \in \mathbb{N}$, we have that 
$$
\sum_{N_{m}=1}^{\infty}{\cdots \sum_{N_1=1}^{N_2}{\frac{1}{N_{m}^{2} \cdots N_{1}^{2}}}}
=\frac{(-1)^{m+1}2 \left(2^{2m-1}-1\right)B_{2m}\pi^{2m}}{(2m)!}
=\left(2-\frac{1}{2^{2(m-1)}}\right)\zeta(2m)
$$
or identically (from Theorem \ref{Theorem 4.1}), 
$$
\sum_{\sum{i.y_{k,i}}=m}{\prod_{i=1}^{m}{\frac{1}{(y_{k,i})!i^{y_{k,i}}} \left(\zeta(2i)\right)^{y_{k,i}}}}
=\frac{(-1)^{m+1}2 \left(2^{2m-1}-1\right)B_{2m}\pi^{2m}}{(2m)!}
=\left(2-\frac{1}{2^{2(m-1)}}\right)\zeta(2m)
.$$
\end{theorem}
\begin{proof}
In \cite{Schneider}, the following relation was proven but in another notation, 
$$
\sum_{1 \leq N_1 \leq \cdots \leq N_m}{\frac{1}{N_{m}^{2} \cdots N_{1}^{2}}}
=\left(\frac{2^{2m-1}-1}{2^{2m-2}}\right)\zeta(2m)
.$$
Hence, 
$$
\sum_{N_{m}=1}^{\infty}{\cdots \sum_{N_1=1}^{N_2}{\frac{1}{N_{m}^{2} \cdots N_{1}^{2}}}}
=\left(\frac{2^{2m-1}-1}{2^{2m-2}}\right)\zeta(2m)
=\left(2-\frac{1}{2^{2(m-1)}}\right)\zeta(2m)
.$$
Euler proved that, for $m \geq 1$,  
$$
B_{2m}=\frac{(-1)^{m+1}2(2m)!}{(2\pi)^{2m}}\zeta(2m) \,\,\,\,\,\,\,\, or \,\,\,\,\,\,\,\,
\zeta(2m)=\frac{(-1)^{m+1}(2\pi)^{2m}}{2(2m)!}B_{2m}
.$$
Hence, by substituting, we get 
$$
\sum_{N_{m}=1}^{\infty}{\cdots \sum_{N_1=1}^{N_2}{\frac{1}{N_{m}^{2} \cdots N_{1}^{2}}}}
=\frac{(-1)^{m+1}2 \left(2^{2m-1}-1\right)B_{2m}\pi^{2m}}{(2m)!}
.$$
The proof of the first equation is completed. \\
Applying Theorem \ref{Theorem 4.1}, we get the second equation. 
\end{proof}
\begin{corollary}
For any $m \in \mathbb{N}$, we have that 
$$
\sum_{\sum{i.y_{k,i}}=m}{\prod_{i=1}^{m}{\frac{(-1)^{y_{k,i}}}{(y_{k,i})!} \left(\frac{B_{2ip}}{(2i)(2ip)!}\right)^{y_{k,i}}}}
=\left(\frac{1}{2^{2m-1}}-1\right)\frac{B_{2m}}{(2m)!}
$$
\end{corollary}
\begin{proof}
By applying Theorem \ref{Theorem 4.3} with $p=1$ to Theorem \ref{Conjecture 4.1}, we obtain the theorem. 
\end{proof}
We will use this to prove that this recurrent harmonic series (or recurrent $p$-series) with $2p=2$ will converge to 2 as the number of summations $m$ goes to infinity. 
\begin{theorem} \label{Theorem 4.4} 
For any $m \in \mathbb{N}$, we have that 
$$
\lim_{m \to \infty}{{\left(\sum_{N_{m}=1}^{\infty}{\cdots \sum_{N_1=1}^{N_2}{\frac{1}{N_{m}^{2} \cdots N_{1}^{2}}}}\right)}}
=2
.$$
\end{theorem}
\begin{proof}
It is known that $\lim_{m \to \infty} \zeta(2m)=1$. By applying Theorem \ref{Conjecture 4.1}, 
$$
\lim_{m \to \infty}{{\left(\sum_{N_{m}=1}^{\infty}{\cdots \sum_{N_1=1}^{N_2}{\frac{1}{N_{m}^{2} \cdots N_{1}^{2}}}}\right)}}
=\lim_{m \to \infty}{\left(2-\frac{1}{2^{2(m-1)}}\right)} \times \lim_{m \to \infty}{\zeta(2m)}
=2
.$$
\end{proof}
\begin{example}
{\em For $m=4$, we have }
\begin{dmath*}
\sum_{N_4=1}^{\infty}{\sum_{N_3=1}^{N_4}{\sum_{N_2=1}^{N_3}{\sum_{N_1=1}^{N_2}{\frac{1}{N_4^2 N_3^2 N_2^2 N_1^2}}}}}
=\frac{1}{24}\left(\sum_{N=1}^{\infty}{\frac{1}{N^2}}\right)^{4}
+\frac{1}{4}\left(\sum_{N=1}^{\infty}{\frac{1}{N^2}}\right)^{2}\left(\sum_{N=1}^{\infty}{\frac{1}{N^4}}\right)
+\frac{1}{3}\left(\sum_{N=1}^{\infty}{\frac{1}{N^2}}\right)\left(\sum_{N=1}^{\infty}{\frac{1}{N^6}}\right)
+\frac{1}{8}\left(\sum_{N=1}^{\infty}{\frac{1}{N^4}}\right)^{2}
+\frac{1}{4}\left(\sum_{N=1}^{\infty}{\frac{1}{N^8}}\right)
=\frac{127\pi^8}{604800}
{=\left(2-\frac{1}{2^{2(3)}}\right)\zeta(8)
\approx 1.992466004
.}\end{dmath*}
\end{example}
Similarly, we will use this to show that the sum (over all non-negative values of $m$) of the  recurrent harmonic series with $2p=2$ will diverge. 
\begin{theorem} \label{Theorem 4.5} 
We have that, 
$$
\sum_{m=0}^{\infty}{\left(\sum_{N_{m}=1}^{\infty}{\cdots \sum_{N_1=1}^{N_2}{\frac{1}{N_{m}^{2} \cdots N_{1}^{2}}}}\right)}
\to \infty
.$$
\end{theorem}
\begin{proof}
Applying Theorem \ref{Conjecture 4.1},
$$ 
\sum_{m=0}^{\infty}{\left(\sum_{N_{m}=1}^{\infty}{\cdots \sum_{N_1=1}^{N_2}{\frac{1}{N_{m}^{2} \cdots N_{1}^{2}}}}\right)}
=\sum_{m=0}^{\infty}{\left(2-\frac{1}{2^{2(m-1)}}\right)\zeta(2m)}
.$$
For $m=0$, we have $\left(2-\frac{1}{2^{2(m-1)}}\right)\zeta(2m)=(2-4)(-1/2)=1$. 
Knowing that $\zeta(2m) \geq 1$ for $m \geq 1$ and noticing the following identity for $m \geq 1$, 
$$
1 \leq \left(2-\frac{1}{2^{2(m-1)}}\right) \leq 2
.$$
Hence, for $m \geq 0$, 
$$
\left(2-\frac{1}{2^{2(m-1)}}\right)\zeta(2m) \geq 1
.$$
Thus, 
$$
\lim_{n \to \infty} {\sum_{m=0}^{n}{\left(\sum_{N_{m}=1}^{\infty}{\cdots \sum_{N_1=1}^{N_2}{\frac{1}{N_{m}^{2} \cdots N_{1}^{2}}}}\right)}} 
\geq \lim_{n \to \infty} {\sum_{m=0}^{n}{1}} 
=\infty
.$$
Hence, this sums is infinite. 
\end{proof}

\bibliographystyle{apa}
\bibliography{Recurrent_Sums}

\end{document}